\documentclass{amsart}
\usepackage{amssymb,mathrsfs}

\def\bC {\mathbf{C}}
\def\bD {\mathbf{D}}

\def\bh {\mathbf{h}}
\def\bH {\mathbf{H}}

\def\bR {\mathbf{R}}

\def\fH {\mathfrak{H}}
\def\fS {\mathfrak{S}}

\def\cD {\mathcal{D}}

\def\cF {\mathcal{F}}

\def\cH {\mathcal{H}}

\def\cL {\mathcal{L}}

\def\cP {\mathcal{P}}
\def\cQ {\mathcal{Q}}

\def\cS {\mathcal{S}}
\def\cT {\mathcal{T}}
\def\cV {\mathcal{V}}

\def\de {{\delta}}
\def\eps {{\epsilon}}

\def\th {{\theta}}

\def\L {{\Lambda}}
\def\si {{\sigma}}

\def\d {{\partial}}
\def\grad {{\nabla}}
\def\Dlt {{\Delta}}

\def\rstr {{\big |}}

\def\la {\langle}
\def\ra {\rangle}

\newcommand{\Div}{\operatorname{div}}

\newcommand{\Tr}{\operatorname{trace}}

\newcommand{\Lip}{\operatorname{Lip}}

\newcommand{\MKp}{\operatorname{dist_{MK,p}}}

\newcommand{\MKd}{\operatorname{dist_{MK,2}}}
\newcommand{\Op}{\operatorname{OP}}

\newcommand{\ba}{\begin{aligned}}
\newcommand{\ea}{\end{aligned}}

\newcommand{\be}{\begin{equation}}
\newcommand{\ee}{\end{equation}}

\newcommand{\lb}{\label}

\newtheorem{Thm}{Theorem}[section]
\newtheorem{Rmk}[Thm]{Remark}

\newtheorem{Lem}[Thm]{Lemma}
\newtheorem{Def}[Thm]{Definition}

\begin{document}

\title[Mean Field and Classical Limits]{On the Mean Field and Classical Limits\\ of Quantum Mechanics}

\author[F. Golse]{Fran\c cois Golse}
\address[F.G.]{Ecole polytechnique, CMLS, 91128 Palaiseau Cedex, France}
\email{francois.golse@polytechnique.edu}

\author[C. Mouhot]{Cl\'ement Mouhot}
\address[C.M.]{University of Cambridge, DPMMS, Wilberforce Road, Cambridge CB3 0WB, United Kingdom}
\email{C.Mouhot@dpmms.cam.ac.uk}

\author[T. Paul]{Thierry Paul}
\address[T.P.]{CNRS and Ecole polytechnique, CMLS, 91128 Palaiseau Cedex, France}
\email{thierry.paul@polytechnique.edu}

\begin{abstract}
The main result in this paper is a new inequality bearing on solutions of the $N$-body linear Schr\"odinger equation and of the mean field Hartree equation. This inequality implies that the mean field limit of the quantum mechanics of $N$ 
identical particles is uniform in the classical limit and provides a quantitative estimate of the quality of the approximation. This result applies to the case of $C^{1,1}$ interaction potentials. The quantity measuring the approximation of the 
$N$-body quantum dynamics by its mean field limit is analogous to the Monge-Kantorovich (or Wasserstein) distance with exponent $2$. The inequality satisfied by this quantity is reminiscent of the work of Dobrushin on the mean field limit 
in classical mechanics [Func. Anal. Appl. \textbf{13} (1979), 115--123]. Our approach of this problem is based on a direct analysis of the $N$-particle Liouville equation, and avoids using techniques based on the BBGKY hierarchy or on 
second quantization.
\end{abstract}

\keywords{Schr\"odinger equation, Hartree equation, Liouville equation, Vlasov equation, Mean field limit, Classical limit, Monge-Kantorovich distance}

\subjclass{82C10, 35Q55 (82C05,35Q83)}

\maketitle

\rightline{\textit{In memory of Louis Boutet de Monvel (1941--2014)}}

\section{Statement of the problem}


In nonrelativistic quantum mechanics, the dynamics of $N$ identical particles of mass $m$ in $\bR^d$ is described by the linear Schr\"odinger equation
$$
i\hbar\d_t\Psi=-\frac{\hbar^2}{2m}\sum_{k=1}^N\Dlt_{x_k}\Psi+\tfrac12\sum_{k,l=1}^NV(x_k-x_l)\Psi\,,
$$
where the unknown is $\Psi\equiv\Psi(t,x_1,\ldots,x_N)\in\bC$, the $N$-particle wave function, while $x_1,x_2,\ldots,x_N$ designate the positions of the $1$st, $2$nd,\dots, $N$th particle. The interaction between the $k$th and $l$th particles
is given by the potential $V$, a real-valued measurable function defined a.e. on $\bR^d$, such that
\be\lb{EvenV}
V(z)=V(-z)\,,\quad\hbox{ for a.e. }z\in\bR^d\,.
\ee

Denoting the macroscopic length scale by $L\!>\!0$, we define a time scale $T\!>\!0$ such that the total interaction energy of the typical particle with the $N-1$ other particles is of the order of $m(L/T)^2$. With the dimensionless space and time 
variables defined as
$$
\hat x:=x/L\,,\qquad\hat t:=t/T\,,
$$
the interaction potential is scaled as 
$$
\hat V(\hat z):=\frac{NT^2}{mL^2}V(z)\,.
$$
In terms of the dimensionless parameter
$$
\eps:=\hbar T/mL^2
$$
and the new unknown
$$
\hat\Psi(\hat t,\hat x_1,\ldots,\hat x_N):=\Psi(t,x_1,\ldots,x_N)\,,
$$
the Schr\"odinger equation becomes
$$
i\d_{\hat t}\hat\Psi=-\tfrac12\eps\sum_{k=1}^N\Dlt_{\hat x_k}\hat\Psi+\frac1{2N\eps}\sum_{k,l=1}^N\hat V(\hat x_k-\hat x_l)\hat\Psi\,.
$$
In the present paper, we obtain a new estimate for solutions of this Schr\"odinger equation which is of particular interest in the asymptotic regime where
$$
\eps\ll 1\hbox{ (classical limit) }\quad\hbox{ and }\quad N\gg 1\hbox{ (mean field limit)}\,.
$$
Henceforth we drop all hats on rescaled quantities and consider the Cauchy problem
\be\lb{CPSchr}
\left\{
\ba
{}&i\d_t\Psi_{\eps,N}=-\tfrac12\eps\sum_{k=1}^N\Dlt_{x_k}\Psi_{\eps,N}+\frac1{2N\eps}\sum_{k,l=1}^NV(x_k-x_l)\Psi_{\eps,N}\,,
\\
&\Psi_{\eps,N}\rstr_{t=0}=\Psi^{in}_{\eps,N}\,.
\ea
\right.
\ee

While the discussion above applies to all types of particles, until the end of this section we restrict our attention to the case of bosons, i.e. to the case where the wave function $\Psi_{\eps,N}$ is a symmetric function of the space variables $x_1,\ldots,x_N$.

A typical example of relevant initial data for (\ref{CPSchr}) is
$$
\Psi_{\eps,N}^{in}(x_1,\ldots,x_N)=(\pi\eps)^{-dN/4}\prod_{j=1}^Ne^{-(x_j-q)^2/2\eps}e^{ip\cdot x_j/\eps}\,,
$$
where $p,q\in\bR^d$ are parameters. In other words, in this example, the initial wave function is the $N$-fold tensor product of Gaussian wave functions with width $\sqrt{\eps}$ and oscillations at frequency $O(1/\eps)$.

\subsection{The mean field limit}


The mean field limit is the asymptotic regime where $N\to\infty$, with $\eps>0$ fixed. Set the initial data in (\ref{CPSchr}) to be
$$
\Psi^{in}_{\eps,N}(x_1,\ldots,x_N):=\prod_{k=1}^N\psi^{in}_\eps(x_k)\quad\hbox{ with }\quad\int_{\bR^d}|\psi^{in}_\eps(x)|^2dx=1\,,
$$
and let $\Psi_{\eps,N}$ be the solution of the Cauchy problem (\ref{CPSchr}) --- which exists for all times provided that $V$ is such that
$$
-\tfrac12\eps\sum_{k=1}^N\Dlt_{x_k}+\frac1{2N\eps}\sum_{k,l=1}^NV(x_k-x_l)
$$
has a self-adjoint extension as an unbounded operator on $L^2((\bR^d)^N)$. Under various assumptions on $\psi^{in}$ and $V$, it is known that, for each $t\in\bR$,
$$
\int_{(\bR^d)^{N-1}}\Psi_{\eps,N}(t,x,z_2,\ldots,z_N)\overline{\Psi_{\eps,N}(t,y,z_2,\ldots,z_N)}dz_2\ldots dz_N\to\psi_\eps(t,x)\overline{\psi_\eps(t,y)}
$$
in some appropriate sense as $N\to\infty$, where $\psi_\eps$ is the solution of the Hartree equation
\be\lb{CPH}
\left\{
\ba
{}&i\d_t\psi_\eps=-\tfrac12\eps\Dlt_{x}\psi_\eps+\frac1{\eps}\psi_\eps(t,x)\int_{\bR^d}V(x-z)|\psi_\eps(t,z)|^2dz\,,
\\
&\psi_\eps\rstr_{t=0}=\psi^{in}_\eps\,.
\ea
\right.
\ee
See \cite{Spohn, BGM, ErdY, AdaFGTeta, ErdYSch1, ErdYSch2, FrohKnow, RodSch, Pickl1, Pickl2, PicklKnow} for various results in this direction, obtained under different assumptions on the regularity of the interaction potential $V$. Most of
the physically relevant particle interactions, especially the case where $V$ is the Coulomb potential, are covered by these results, but not always with a quantitative error estimate.

\subsection{The classical limit}


The classical limit is the asymptotic regime where $\eps\to 0$ while $N$ is kept fixed in (\ref{CPSchr}) --- or simply $\eps\to 0$ in (\ref{CPH}). The formalism of the \textit{Wigner transform} is perhaps the most convenient way to formulate this limit. 
Given $\Phi\equiv\Phi(X)\in\bC$, an element of $L^2(\bR^n)$, its Wigner transform at scale $\eps$ is
$$
W_\eps[\Phi](X,\Xi):=\tfrac1{(2\pi)^n}\int_{\bR^n}\Phi\left(X+\tfrac12\eps Y\right)\overline{\Phi\left(X-\tfrac12\eps Y\right)}e^{-i\Xi\cdot Y}dY\,.
$$
Assume that the initial data in (\ref{CPSchr}) is a family $\Psi^{in}_{\eps,N}$ such that
$$
W_\eps[\Psi^{in}_{\eps,N}]\to F^{in}_N\quad\hbox{ in }\cS'((\bR^d\times\bR^d)^N)\hbox{  as }\eps\to 0\,.
$$
Then, for all $t\in\bR$, the family of solutions $\Psi_{\eps,N}$ of the Cauchy problem (\ref{CPSchr}) satisfies
$$
W_\eps[\Psi_{\eps,N}(t,\cdot)]\to F_N(t,\cdot,\cdot)\quad\hbox{ in }\cS'((\bR^d\times\bR^d)^N)\hbox{  as }\eps\to 0\,,
$$
where $F_N\equiv F_N(t,x_1,\ldots,x_N,\xi_1,\ldots,\xi_N)\ge 0$ is the solution of the following Cauchy problem for the $N$-body Liouville equation of classical mechanics
\be\lb{NLiouv}
\left\{
\ba
{}&\d_tF_N+\sum_{k=1}^N\xi_k\cdot\grad_{x_k}F_N-\frac1N\sum_{k,l=1}^N\grad V(x_k-x_l)\cdot\grad_{\xi_k}F_N=0\,,
\\
&F_N\rstr_{t=0}=F_N^{in}\,.
\ea
\right.
\ee
The classical limit of the Hartree equation (\ref{CPH}) can be formulated similarly. Assume that the initial data in (\ref{CPH}) is a family $\psi^{in}_\eps$ such that
$$
W_\eps[\psi^{in}_\eps]\to f^{in}\quad\hbox{ in }\cS'(\bR^d\times\bR^d)\hbox{  as }\eps\to 0\,.
$$
Then, for all $t\in\bR$, the family of solutions $\psi_\eps$ of the Hartree equation (\ref{CPH}) satisfies
$$
W_\eps[\psi_\eps(t,\cdot)]\to f(t,\cdot,\cdot)\quad\hbox{ in }\cS'(\bR^d\times\bR^d)\hbox{  as }\eps\to 0\,,
$$
where $f\equiv f(t,x,\xi)\ge 0$ is the solution of the following Cauchy problem for the Vlasov equation of classical mechanics with interaction potential $V$:
\be\lb{Vlas}
\left\{
\ba
{}&\d_tf+\xi\cdot\grad_xf-\left(\int_{\bR^d}\grad V(x-z)f(t,z)dz\right)\cdot\grad_\xi f=0\,,
\\
&f\rstr_{t=0}=f^{in}\,.
\ea
\right.
\ee
See \cite{PLLTP, GMMP} for results on the classical limit of quantum mechanics involving the Wigner transform.

\subsection{The mean field limit in classical mechanics}\lb{SS-MFCM}


There is also a notion of mean field limit in classical mechanics, which can be formulated as follows. Assume that the initial data in (\ref{NLiouv}) is
$$
F^{in}_N(x_1,\ldots,x_N,\xi_1,\ldots,\xi_N)=\prod_{k=1}^Nf^{in}(x_k,\xi_k)\,,
$$
where $f^{in}$ is a probability density on $\bR^d\times\bR^d$. Under various assumptions on the potential $V$, the solution $F_N$ of (\ref{NLiouv}) satisfies
$$
\int_{(\bR^d\times\bR^d)^{N-1}}F_N(t,x,x_2,\ldots,x_N,\xi,\xi_2,\ldots,\xi_N)dx_2d\xi_2\ldots dx_Nd\xi_N\to f(t,x,\xi)
$$
in some appropriate sense as $N\to\infty$, where $f$ is the solution of the Cauchy problem for the Vlasov equation (\ref{Vlas}). See \cite{NeunWick, BraunHepp, Dobrush} for the missing details. All these references address the problem of 
the mean field limit in terms of the empirical measure of the $N$-particle system. Typically these results cover the case where $V\in C^{1,1}(\bR)$, but the case of a Coulomb, or Newtonian interaction remains open at the time of this writing. 
For a formulation of the same results in terms of the BBGKY hierarchy, see \cite{FGMouRi, MischMouWenn}. 

Yet, there has been some recent progress on the case of singular potentials. In \cite{HaurJab}, the mean field limit has been established for interaction potentials with a singularity at the origin that is weaker than that of the Coulomb
potential. Another approach to the mean field limit in the case of singular interaction involves a truncated variant of the potential with a cutoff parameter $\eta\equiv\eta(N)>0$ assumed to vanish as the number of particles $N\to\infty$: see
\cite{HaurJab,Laza,LazaPickl} for the most recent results in that direction. This cutoff parameter can be thought of as being of the order of the size of the interacting particles, as explained in \cite{Laza}.

\bigskip
The situation described above can be summarized in the following diagram: the horizontal arrows correspond to the mean field limit, while the vertical arrows correspond to the classical limit.

\bigskip

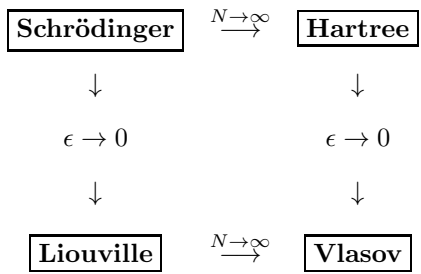
\begin{center}
\begin{tabular}{ccc}
{\fbox{\bf Schr\"odinger}}& {$\stackrel{N\to\infty}{\longrightarrow}$}& {\fbox{\bf Hartree}} 
\\ [3mm]
{$\downarrow$} & & {$\downarrow$} 
\\[3mm]
{${\eps\to0}$}&& {${\eps\to 0}$} 
\\ [3mm]
$\downarrow$&& $\downarrow$ 
\\ [3mm]
{\fbox{\bf Liouville}}& {$\stackrel{N\to\infty}{\longrightarrow}$}&{\fbox {\bf Vlasov}} 
\end{tabular}
\end{center}

\bigskip
However, these various limits are established by very different methods. The classical limits of either the $N$-body Schr\"odinger equation or of the Hartree equation are obtained by a compactness argument and the uniqueness of the 
solution of the Cauchy problems (\ref{NLiouv}) or (\ref{Vlas}). Error estimates for these limits require rather stringent assumptions on the regularity of the potential $V$ and on the type of initial wave or distribution functions considered. The 
mean field limit in quantum mechanics (the upper horizontal arrow) comes from trace norm estimates on the infinite hierarchy of equations obtained from the BBGKY hierarchy in the large $N$ limit. The trace norm is the quantum analogue 
of the total variation norm on the probability measures appearing in the classical setting. In general, the total variation norm is not convenient in the context of the mean field limit, since it does not capture the distance between neighboring 
point particles. This suggests that the trace norm is not appropriate to obtain controls on the large $N$ (mean field) limit which remain uniform in the vanishing $\eps$ (classical) limit.

Another notable difficulty with this problem is that the mean field limit in classical mechanics is obtained by proving the weak convergence of the $N$-particle empirical measure in the single-particle phase space to the solution of the Vlasov 
equation. Since there does not seem to be any natural analogue of the notion of empirical measure for a quantum $N$-particle system, the analogy between the quantum and the classical mean field limits is far from obvious. 

Our main result, stated as Theorem \ref{T-UQMF} below, is a new quantitative estimate for the mean field limit $N\to\infty$ of quantum mechanics which is uniform in the classical limit $\eps\to 0$.


\section{Main result}


Let $d$ be a positive integer. Henceforth we set $\fH:=L^2(\bR^d)$, and $\fH_N:=\fH^{\otimes N}\simeq L^2((\bR^d)^N)$ for each $N\ge 1$. We designate by $\cL(\fH)$ the algebra of bounded linear operators on the Hilbert space $\fH$.
We denote by $\cD(\fH_N)$ the set of operators $A\in\cL(\fH_N)$ such that
$$
A=A^*\ge 0\,,\qquad\hbox{ and }\Tr(A)=1\,.
$$

We are concerned with the $N$-body Schr\"odinger equation written in terms of density matrices, i.e. the von Neumann equation
\be\lb{NBodyQuantRho}
\left\{
\ba
{}&i\d_t\rho_{\eps,N}=\left[-\tfrac12\eps\sum_{k=1}^N\Dlt_k+\frac1{2N\eps}\sum_{k,l=1}^NV_{kl},\rho_{\eps,N}\right]\,,
\\
&\rho_{\eps,N}\rstr_{t=0}=\rho_{\eps,N}^{in}\,,
\ea
\right.
\ee
where $\rho_{\eps,N}(t)\in\cD(\fH_N)$, while
$$
\Dlt_k:=I_\fH^{\otimes(k-1)}\otimes\Dlt\otimes I_{\fH}^{\otimes(N-k)}\,,
$$
and
\be\lb{DefVjk}
(V_{jk}\psi)(x_1,\ldots,x_N):=V(x_k-x_j)\psi(x_1,\ldots,x_N)\,,\quad\hbox{ for each }\psi\in\fH_N\,.
\ee
The notation $I_\fH$ designates the identity on the Hilbert space $\fH$.

We shall everywhere restrict our attention to symmetric density matrices, corresponding to indistinguishable particles. In other words, we assume that
\be\lb{SymDens(t)}
\tau_\si\rho_{\eps,N}(t)\tau^*_\si=\rho_{\eps,N}(t)
\ee
for each $t\in\bR$ and each $\si\in\fS_N$, where $\tau_\si$ is the unitary operator defined on $L^2((\bR^d)^N)$ by the formula
\be\lb{DefTauSi}
\tau_\si\Phi(x_1,\ldots,x_N):=\Phi(x_{\si^{-1}(1)},\ldots,x_{\si^{-1}(N)})\,.
\ee
One easily checks that the condition
\be\lb{SymDensIn}
\tau_\si\rho_{\eps,N}^{in}\tau^*_\si=\rho_{\eps,N}^{in}
\ee
implies that (\ref{SymDens(t)}) holds for each $t\in\bR$, since 
$$
\left[\tau_\si,-\tfrac12\eps\sum_{k=1}^N\Dlt_k+\frac1{2N\eps}\sum_{k,l=1}^NV_{kl}\right]=0\quad\hbox{ for each }\si\in\fS_N\,.
$$

On the other hand, we consider the corresponding mean field equation, i.e. the Hartree equation written in terms of the density matrix $\rho_\eps(t)\in\cD(\fH)$
\be\lb{HarRho}
\left\{
\ba
{}&i\d_t\rho_\eps=\left[-\tfrac12\eps\Dlt+\frac1{\eps}V_{\rho_\eps},\rho_\eps\right]\,,
\\
&\rho_\eps\rstr_{t=0}=\rho_\eps^{in}\,,
\ea
\right.
\ee
where $V_{\rho_\eps}$ designates both the function
$$
V_{\rho_\eps}(t,x):=\int_{\bR^d}V(x-z)\rho_\eps(t,z,z)dz
$$
and the time-dependent multiplication operator defined on $\fH$ by 
$$
(V_{\rho_\eps}\psi)(t,x):=V_{\rho_\eps}(t,x)\psi(x)\,.
$$

Next we formulate the mean field limit in terms of density operators. For each $N$-particle density operator $\rho_N\in\cD(\fH_N)$, we define its first $n$-particle marginal density operator, denoted by $\rho^\mathbf{n}_N$ for each integer
$n$ such that $1\le n\le N$, by the following conditions:
$$
\left\{
\ba
{}&\,\rho^\mathbf{n}_N\in\cD(\fH_n)\,,\quad\hbox{ and }
\\
&\Tr_{\fH_n}(A\rho^\mathbf{n}_N)=\Tr_{\fH_N}((A\otimes I_{\fH_{N-n}})\rho_N)\quad\hbox{ for each }A\in\cL(\fH_n)\,.
\ea
\right.
$$
In the mean field limit, i.e. for $N\to\infty$ while $\eps>0$ is kept fixed, one expects that the sequence $\rho^\mathbf{1}_{\eps,N}$ of first marginals of the density operators $\rho_{\eps,N}$ solutions of (\ref{NBodyQuantRho}) converges in 
some topology to the solution $\rho_\eps$ of (\ref{HarRho}), provided that $\rho_{\eps,N}^{in}$ approaches $(\rho_\eps^{in})^{\otimes N}$ in some appropriate sense. The difference between $\rho^\mathbf{1}_{\eps,N}$ and $\rho_\eps$ 
is measured in terms of a quantity analogous to the Monge-Kantorovich distance used in the context of optimal transport.

\smallskip
First we define the notion of coupling between two density operators.

\begin{Def}\lb{D-DefQ}
Let $d$ be a positive integer and let $\fH:=L^2(\bR^d)$. For each $\rho,\overline{\rho}\in\cD(\fH)$, let $\cQ(\rho,\overline{\rho})$ be the set of $R\in\cD(\fH_2)$ such that
$$
\left\{
\ba
{}&\Tr_{\fH_2}((A\otimes I_\fH)R)=\Tr_\fH(A\rho)
\\
&\Tr_{\fH_2}((I_\fH\otimes A)R)=\Tr_\fH(A\overline{\rho})
\ea
\right.
$$
for each $A\in\cL(\fH)$.
\end{Def}

\smallskip
Next we define two unbounded operators on $\fH_2\simeq L^2(\bR^d\times\bR^d)$, as follows:
\be\lb{DefPQ}
\left\{
\ba
{}&(Q\psi)(x_1,x_2):=(x_1-x_2)\psi(x_1,x_2)\,,
\\
&(P\psi)(x_1,x_2):=-i\eps(\grad_{x_1}-\grad_{x_2})\psi(x_1,x_2)\,,
\ea
\right.
\ee
so that
\be\lb{P*PQ*Q}
(P^*P\psi)(x_1-x_2)=-\eps^2(\Div_{x_1}-\Div_{x_2})(\grad_{x_1}-\grad_{x_2})\psi(x_1,x_2)\,.
\ee

\smallskip
The quantum analogue of the Monge-Kantorovich distance with exponent $2$ is defined as follows. For the definition of Monge-Kantorovich distances, also called Wasserstein distances, see formula (\ref{DefMKp}) in section \ref{S-MFCM}, 
or chapter 7 in \cite{VillaniTOT}, or chapter 6 in \cite{VillaniTOT2}.

\begin{Def}\lb{D-DefMKeps}
For each $\rho,\overline{\rho}\in\cD(\fH)$, we set
$$
MK^\eps_2(\rho,\overline{\rho}):=\inf_{R\in\cQ(\rho,\overline{\rho})}\Tr_{\fH_2}((Q^*Q+P^*P)R)^{1/2}
$$
with the following convention:
$$
\Tr_{\fH_2}((Q^*Q+P^*P)R):=\Tr_{\fH_2}(R^{1/2}(Q^*Q+P^*P)R^{1/2})
$$
if $R^{1/2}(Q^*Q+P^*P)R^{1/2}$ is a trace-class operator, and
$$
\Tr_{\fH_2}((Q^*Q+P^*P)R):=+\infty
$$
otherwise.
\end{Def}

\smallskip
The quantity $MK^\eps_2$ is not a distance on $\cD(\fH)$. In fact, for each density operator $\rho$ on $\fH$, one has $MK^\eps_2(\rho,\rho)>0$ (see formula (\ref{LBMKeps}) below). However, $MK^\eps_2$ can be compared with the 
Monge-Kantorovich distance with exponent $2$ (see formula (\ref{DefMKp}) below with $p=2$ for a definition of this distance, and Theorem \ref{T-MKeps} (2) below for more information on this comparison), at least for a certain class 
of operators and in the limit as $\eps\to 0$. 

\smallskip
Henceforth, we denote by $\cP(\bR^d)$ the set of Borel probability measures on $\bR^d$. For each $q>0$, we define
$$
\cP_q(\bR^d):=\left\{\mu\in\cP(\bR^d)\hbox{ s.t. }\int_{\bR^d}|x|^q\mu(dx)<\infty\right\}\,.
$$

\begin{Thm}[Properties of $MK^\eps_2$]\lb{T-MKeps} Let $d$ be a positive integer and let $\fH:=L^2(\bR^d)$. For each $\rho,\overline{\rho}\in\cD(\fH)$ and each $\eps>0$, one has
\be\lb{LBMKeps}
MK^\eps_2(\rho,\overline{\rho})^2\ge2d\eps\,.
\ee
(1) Let $\eps>0$ and let $\rho^\eps_1$ and $\rho^\eps_2$ be T\"oplitz operators at scale $\eps$ on $L^2(\bR^d)$ with symbols $(2\pi\eps)^d\mu_1$ and $(2\pi\eps)^d\mu_2$, where $\mu_1,\mu_2\in\cP_2(\bR^{2d})$. Then
$$
\ba
MK^\eps_2(\rho^\eps_1,\rho^\eps_2)^2&\le\inf_{\mu\in\Pi(\mu_1,\mu_2)}\Tr_{\fH\otimes\fH}((Q^*Q+P^*P)\Op^T_\eps((2\pi\eps)^{2d}\mu))
\\
&=\MKd(\mu_1,\mu_2)^2+2d\eps\,.
\ea
$$
(2) Let $\rho^\eps_1$ and $\rho^\eps_2\in\cD(\fH)$, with Husimi transforms at scale $\eps$ denoted respectively $\tilde W_\eps[\rho^\eps_1]$ and $\tilde W_\eps[\rho^\eps_2]$. Then
$$
MK^\eps_2(\rho^\eps_1,\rho^\eps_2)^2\ge\MKd(\tilde W_\eps[\rho^\eps_1],\tilde W_\eps[\rho^\eps_2])^2-2d\eps\,.
$$
Assume further that the Wigner transforms at scale $\eps$ of $\rho^\eps_1$ and $\rho^\eps_2$, denoted respectively $W_\eps[\rho^\eps_1]$ and $W_\eps[\rho^\eps_2]$, converge in $\cS'(\bR^{2d})$ to Wigner measures denoted 
respectively $w_1$ and $w_2$ as $\eps\to 0$. Then
$$
\MKd(w_1,w_2)\le\varliminf_{\eps\to 0}MK^\eps_2(\rho^\eps_1,\rho^\eps_2)\,.
$$
\end{Thm}

The definitions and basic properties of T\"oplitz operators, Wigner and Husimi functions are recalled in Appendix \ref{A-Toeplitz}. 

Statement (2) in Theorem \ref{T-MKeps} implies in particular that the quantity $MK^\eps_2$ is not vanishing for all density matrices as $\eps\to 0^+$. This property is obviously essential; otherwise, the quantity $MK^\eps_2$ would not
be of much practical interest for controlling the error in the mean field limit. Statement (1) in Theorem \ref{T-MKeps} will be used in choosing the initial quantum states to which our error estimate for the mean field limit will apply.

\smallskip
The main result in this paper is the following theorem.

\begin{Thm}\lb{T-UQMF}
Let $d$ be a positive integer. For each $\eps>0$ and each integer $N>1$, let $\rho_{\eps,N}^{in}\in\cD(L^2((\bR^d)^N))$ satisfy (\ref{SymDensIn}) and let $\rho_\eps^{in}\in\cD(L^2(\bR^d))$. Let $t\mapsto\rho_{\eps,N}(t)$ be the solution 
of the quantum $N$-body Cauchy problem (\ref{NBodyQuantRho}), and let $t\mapsto\rho_\eps(t)$ be the solution of the quantum mean field Cauchy problem (\ref{HarRho}). 
Then, for each positive integer $n\le N$ and all $t\ge 0$, one has
\be\lb{MainIneq}
\ba
\frac1nMK^\eps_2(\rho_\eps(t)^{\otimes n},\rho^\mathbf{n}_{\eps,N}(t))^2\le\frac{8}N\|\grad V\|_{L^\infty}^2\frac{e^{\L t}-1}{\L}+\frac{e^{\L t}}NMK_2^\eps((\rho_\eps^{in})^{\otimes N},\rho_{\eps,N}^{in})^2
\ea
\ee
where
$$
\L:=3+4\Lip(\grad V)^2\,.
$$

In the particular case where $\rho_{\eps,N}^{in}$ is a T\"oplitz operator at scale $\eps$ with symbol $(2\pi\eps)^{dN}\mu_{\eps,N}^{in}$ while $\rho_\eps^{in}$ is a T\"oplitz operator at scale $\eps$ with symbol $(2\pi\eps)^d\mu_\eps^{in}$,
for each positive integer $n\le N$ and all $t\ge 0$, one has
\be\lb{MainIneq2}
\ba
\frac1nMK^\eps_2(\rho_\eps(t)^{\otimes n},\rho^\mathbf{n}_{\eps,N}(t))^2\le&\left(2d\eps+\frac1N\MKd((\mu_\eps^{in})^{\otimes N},\mu_{\eps,N}^{in})^2\right)e^{\L t}
\\
&+\frac{8n}N\|\grad V\|_{L^\infty}^2\frac{e^{\L t}-1}{\L}\,.
\ea
\ee
\end{Thm}

\smallskip
In particular, if $\rho_\eps^{in}\in\cD(\fH)$ is a T\"oplitz operator at scale $\eps$ and $\rho_{\eps,N}^{in}=(\rho_\eps^{in})^{\otimes N}$, then
\be\lb{MainIneqFact}
\frac1nMK^\eps_2(\rho_\eps(t)^{\otimes n},\rho^\mathbf{n}_{\eps,N}(t))^2\le\left(2d\eps+\frac{8}N\|\grad V\|_{L^\infty}^2\frac{1-e^{-\L t}}{\L}\right)e^{\L t}\,.
\ee
Observe that one cannot deduce the mean field limit of the quantum $N$-body problem (\ref{NBodyQuantRho}) from the bound on $MK^\eps_2(\rho_\eps(t),\rho^\mathbf{1}_{\eps,N}(t))$ obtained in Theorem \ref{T-UQMF} in the case where 
$\eps>0$ is kept fixed, because of the term $2d\eps\exp(\L t)$ on the right hand side of (\ref{MainIneq}).

On the other hand, the mean field limit alone, i.e. for $\eps$ fixed, has been proved by other methods in this case (see \cite{Spohn,BGM}). Moreover, quantitative estimates for that limit for $\eps$ fixed and $N\to\infty$ have been obtained in 
\cite{RodSch, Pickl1, Ammari}. Therefore, only the case where both $N\to\infty$ and $\eps\to 0$ remains to be treated, and the present work answers precisely this question.

Indeed, the estimate in Theorem \ref{T-UQMF}, together with the first lower bound in Theorem \ref{T-MKeps} (2), implies that the mean field limit, i.e. the convergence
$$
\rho^\mathbf{n}_{\eps,N}(t)\to\rho_\eps(t)^{\otimes n}
$$
for each $n\ge 1$ as $N\to\infty$ is uniform as $\eps\to 0$ and over long times intervals, in the following sense. Let $c,c'$ satisfy $0<c<c'<1$, and set
$$
\ba
T(c,\eta,N):=\frac{c}\L\min\left(\ln\frac1\eta,\ln N\right)\,,\quad n(c',\eta,N):=\left[\min\left(\eta^{c'-1},N^{1-c'}\right)\right]\,,
\ea
$$
(where $[x]$ designates the largest integer less than or equal to $x$). Then the quadratic Monge-Kantorovich distance between the Husimi transforms at scale $\eps$ of $\rho^\mathbf{n}_{\eps,N}(t)$ and $\rho^{\otimes n}_\eps(t)$ satisfies
$$
\sup_{1\le n\le n(c',\eta,N)}\sup_{0\le t\le T(c,\eta,N)}\sup_{0<\eps<\eta}\MKd(\tilde W_\eps[\rho_\eps(t)^{\otimes n}],\tilde W_\eps[\rho^\mathbf{n}_{\eps,N}(t)])\to 0
$$
in the limit as $\frac1N+\eta\to 0$.

Earlier works have discussed the mean field limit of the quantum $N$-body problem in the small $\eps$ regime, more precisely, in the case where $\eps=\eps(N)\to 0$ as $N\to\infty$. 

The case $\eps(N)=N^{-1/3}$ is of considerable importance for the mean field limit of systems of $N$ fermions, and has been investigated in \cite{NarnhoSewell,SpohnM2AS}. In the more recent reference \cite{BPS}, the $N$-body problem in 
the fermionic case is compared to the Hartree-Fock equations, with error estimates in the Hilbert-Schmidt and trace norms --- see formulas (2.19)-(2.20)  in Theorem 2.1 of \cite{BPS}. Of course, convergence in either Hilbert-Schmidt or trace 
norm is stronger than the control (\ref{MainIneq}) in terms of the quantity $MK^\eps_2$. On the other hand, at variance with the estimate (\ref{MainIneq}) above, whose right hand side grows exponentially fast in $t$, the bounds (2.19)-(2.20) 
in \cite{BPS} involves a right hand side growing much faster in $t$ --- specifically, of order $\exp(c_1\exp(c_2|t|))$ for some constants $c_1,c_2>0$. Another difference between \cite{BPS} and our work is that the inequality (\ref{MainIneq}) does
not postulate any dependence of $\eps$ in $N$ --- on the contrary, $\eps$ and $N$ are independent throughout the present paper.

In \cite{GMP}, for each sequence $\eps\equiv\eps(N)\to 0$ as $N\to\infty$ and each monokinetic solution of the Vlasov equation (\ref{Vlas}) --- i.e. of the form $f(t,x,\xi)=\rho(t,x)\de(\xi-u(t,x))$ --- Theorem 1.1 gives an asymptotic 
approximation rate for the convergence of the Wigner transform at scale $\eps(N)$ of $\rho^\mathbf{1}_{\eps,N}$ to $f$ in the sense of distributions for $t\in[0,T]$. A priori, the time $T$ and the convergence rate depend on the Vlasov solution 
$f$ and on the sequence $\eps(N)$. On the diagram of section \ref{SS-MFCM}, this result corresponds to the left vertical and bottom horizontal arrow along distinguished sequences $(\eps(N),N)$ over time intervals which may depend on the 
dependence of $\eps$ in terms of $N$. Another approach of the same problem can be found in \cite{PezzoPulvi}: it is proved that each term in the semiclassical expansion as $\eps\to 0$ of the quantum $N$-body problem converges as 
$N\to\infty$ to the corresponding term in the semiclassical expansion of Hartree's equation.

On the contrary, Theorem \ref{T-UQMF} provides a quantitative estimate of the distance between the solution of the Hartree equation and the first marginal of the solution of the quantum $N$-body problem, for a rather general class of 
initial data. This estimate implies that the mean field limit, i.e. the top horizontal arrow, is uniform as $\eps\to 0$, over arbitrary long time intervals. This estimate is the quantum analogue of the Dobrushin estimate \cite{Dobrush} for the 
classical mean field limit --- see section \ref{S-MFCM}.

\smallskip
The new ideas used in the proof of Theorem \ref{T-UQMF} are

\smallskip
\noindent
(a) the use of the quantity $MK^\eps_2$, which behaves well with the T\"oplitz quantization, and can be conveniently compared with the Monge-Kantorovich distance with exponent $2$ on symbols, to which it is obviously analogous;

\noindent
(b) an ``Eulerian'' version of Dobrushin's estimate, which avoids the traditional presentation in terms of particle trajectories as in Dobrushin's original work \cite{Dobrush}, and can therefore be easily adapted to the quantum dynamics;

\noindent
(c) and the adaptation of Dobrushin's estimate to the $N$-particle Liouville equation, thereby avoiding the need of any quantum analogue of the classical notion of $N$-particle empirical measure.

\smallskip
The Eulerian version of Dobrushin's inequality (item (b) on the list above) significantly simplifies the original argument, and allows extending Dobrushin's inequality to Monge-Kantorovich distances with arbitrary finite exponents (see
\cite{Loeper} for the original argument for the Monge-Kantorovich distance with exponent $2$). 

Estimating directly the Monge-Kantorovich distance between the first marginal of the $N$-particle distribution function and the solution of the mean field equation (item (c) on the list above) avoids using the fact that the $N$-particle 
empirical measure is an exact solution of the mean field equation, an important feature in Dobrushin's original approach \cite{Dobrush}. This feature is very peculiar to the mean field limit in classical Hamiltonian mechanics, and we
do not know of any quantum analogue of the notion of $N$-particle empirical measure which would exactly satisfy the mean field quantum dynamics. In other words, the mean field limit in quantum mechanics cannot be reduced to
the continuous dependence of solutions of the quantum mean field equation in terms of their initial data, in some appropriate weak topology.

\smallskip
The outline of the paper is as follows: in the next section, we present items (b)-(c) above on the mean field limit for the classical Liouville equation, leading to the Vlasov equation. The resulting estimate in Theorem \ref{T-MFC} below 
improves earlier quantitative bounds of the same type obtained in \cite{FGMouRi, MischMouWenn}. The proof of the uniform in $\eps$ estimate in Theorem \ref{T-UQMF} for the quantum mean field limit occupies section \ref{S-UQMF}. 
The properties of the quantity $MK^\eps_2$ used in this estimate, stated in Theorem \ref{T-MKeps}, are proved in section \ref{S-MKeps}. The material on T\"oplitz quantization, Wigner and Husimi functions used in the proof of Theorem 
\ref{T-MKeps} is recalled in Appendix \ref{A-Toeplitz}.

The T\"oplitz quantization is important in the analysis presented here, since it connects our main result (Theorem \ref{T-UQMF}) with the Dobrushin proof of the mean field limit for the classical $N$-body problem. For this reason, we 
dedicate our work to the memory of our friend and teacher Louis Boutet de Monvel, in recognition of his great contributions to the theory of T\"oplitz operators.


\section{The Mean Field Limit in Classical Mechanics}\lb{S-MFCM}


As a warm-up, we first discuss the mean field limit for the $N$-body problem in classical mechanics, leading to the Vlasov equation. The approach proposed in \cite{NeunWick, BraunHepp, Dobrush} is based on the fact that the phase-space
empirical measure of a $N$-particle system governed by the Newton equations of classical mechanics is a weak solution of the Vlasov equation (\ref{Vlas}). The estimate of the distance between the $N$-particle and the mean field dynamics
obtained by Dobrushin \cite{Dobrush} can be formulated in terms of propagation of chaos for the sequence of marginals of the $N$-particle distribution, as explained in \cite{FGMouRi,MischMouWenn}.

The approach proposed below bears directly on the $N$-particle distribution, i.e. the solution of the Liouville equation (\ref{NLiouv}), and avoids any reference to the $N$-particle empirical measure. Besides, the core of our argument also avoids
using particle trajectories and is based on a computation formulated exclusively in terms of Eulerian coordinates. For that reason, this approach can be adapted to the quantum problem, at variance with the Dobrushin procedure \cite{Dobrush},
also used in \cite{FGMouRi,MischMouWenn}.

For $\mu,\nu\in\cP(\bR^d)$, we denote by $\Pi(\mu,\nu)$ the set of couplings of $\mu$ et $\nu$, i.e. the set of Borel probability measures $\pi$ on $\bR^d\times\bR^d$ with first and second marginals 
\be\lb{Margi12}
\pi_{1}=\mu\quad\hbox{ and }\quad\pi_2=\nu\,.
\ee
In other words,
\be\lb{DefMargi12}
\iint_{\bR^d\times\bR^d}(\phi(x)+\psi(y))\pi(dxdy)=\int_{\bR^d}\phi(x)\mu(dx)+\int_{\bR^d}\psi(y)\nu(dy)
\ee
for each $\phi,\psi\in C_b(\bR^d)$. The identity (\ref{DefMargi12}) can be used as a definition of the first and second marginals of $\pi$ in (\ref{Margi12}). Finally, we recall the definition of the Monge-Kantorovich distance of exponent $p\ge 1$ 
on $\cP_p(\bR^d)$:
\be\lb{DefMKp}
\MKp(\mu,\nu):=\inf_{\pi\in\Pi(\mu,\nu)}\left(\iint_{\bR^d\times\bR^d}|x-y|^p\pi(dxdy)\right)^{1/p}\,.
\ee

\begin{Thm}\lb{T-MFC}
Assume that $V\in C^2_b(\bR^d)$ satisfies (\ref{EvenV}). Let $f^{in}\in\cP_p(\bR^d\times\bR^d)$ with $p\ge 1$. Let $F_N$ be the solution of the Cauchy problem (\ref{NLiouv}) for the $N$-body Liouville equation with initial data
$$
F^{in}_N:=(f^{in})^{\otimes N}\,,
$$
and let $f$ be the solution of the Cauchy problem for the Vlasov equation (\ref{Vlas}) with initial data $f^{in}$. For each integer $n$ such that $1\le n\le N$, let
$$
F^\mathbf{n}_N(t):=\int F_N(t,dy_{n+1}\ldots dy_Nd\eta_{n+1}\ldots d\eta_N)\in\cP_p(\bR^{dn}\times\bR^{dn})
$$
be the $n$-th marginal of $F_N$ (i.e. the marginal corresponding to the phase space distribution of the $n$ first particles). Then
$$
\ba
\frac1n\MKp(f(t)^{\otimes n},F^\mathbf{n}_N(t))^p\le 2^pK_p\|\grad V\|_{L^\infty}^p\frac{[p/2]+1}{N^{\min(p/2,1)}}\frac{e^{\L_pt}-1}{\L_p}\,,
\ea
$$
where $K_p:=\max(1,p-1)$ and $\L_p:=2K_p(1+2^{p-1}\Lip(\grad V)^p)$.
\end{Thm}

\smallskip
Observe that one obtains an estimate of $\MKp(f(t)^{\otimes n},F^\mathbf{n}_N(t))$ of order $O(N^{-1/2})$ if $1\le p\le 2$, corresponding to the rate predicted by the central limit theorem. 

\medskip
\begin{proof} The proof of Theorem \ref{T-MFC} occupies the remaining part of the present section.

\subsection{The dynamics of couplings}


Let $\pi_N^{in}\in\Pi((f^{in})^{\otimes N},(f^{in})^{\otimes N})$ satisfy
\be\lb{SympiN0}
T_\si\#\pi_N^{in}=\pi_N^{in}\,,\qquad\hbox{ for each }\si\in\fS_N\,,
\ee
where 
$$
\ba
T_\si(x_1,\xi_1,\ldots,x_N,\xi_N,y_1,\eta_1,\ldots,y_N,\eta_N)&
\\
=(x_{\si(1)},\xi_{\si(1)},\ldots,x_{\si(N)},\xi_{\si(N)},y_{\si(1)},\eta_{\si(1)},\ldots,y_{\si(N)},\eta_{\si(N)})&\,.
\ea
$$
It will be convenient to use the following notation
\be\lb{DefXNXiN}
\ba
X_N:=(x_1,\ldots,x_N)\,,\quad\Xi_N:=(\xi_1,\ldots,\xi_N)\,,
\\
Y_N:=(y_1,\ldots,y_N)\,,\quad H_N:=(\eta_1,\ldots,\eta_N)\,.
\ea
\ee

Let $f$ be the solution of the Cauchy problem (\ref{Vlas}) with initial data $f^{in}$, and let
$$
\bH_N^{\rho[f]}(X_N,\Xi_N):=\sum_{j=1}^N\left(\tfrac12|\xi_j|^2+V_{\rho[f]}(x_j)\right)
$$
be the mean field Hamiltonian. On the other hand, let
$$
\cH_N(Y_N,H_N):=\sum_{k=1}^N\tfrac12|\eta_k|^2+\frac1{2N}\sum_{k,l=1}^NV(y_k-y_l)
$$
be the microscopic ($N$-particle) Hamiltonian. Finally, we denote by $\{\cdot,\cdot\}_N$ the Poisson bracket on $(\bR^d\times\bR^d)^N$ defined by 
$$
\{\phi,\psi\}_N:=\sum_{j=1}^N(\grad_{\xi_j}\phi\cdot\grad_{x_j}\psi-\grad_{\xi_j}\psi\cdot\grad_{x_j}\phi)\,.
$$

The following observation is the key to the Eulerian formulation of the Dobrushin type estimates, which we shall adapt to the quantum case.

\begin{Lem}\lb{L-DynCoupl}
Let $t\mapsto\pi_N(t)\in\cP_2((\bR^d\times\bR^d)^N\times(\bR^d\times\bR^d)^N)$ be the solution of the Cauchy problem
$$
\left\{
\ba
{}&\d_t\pi_N+\{\bH_N^{\rho[f(t)]}(X_N,\Xi_N)+\cH_N(Y_N,H_N),\pi_N\}_{2N}=0\,,
\\
&\pi_N\rstr_{t=0}=\pi_N^{in}\,.
\ea
\right.
$$
Then
$$
\pi_N(t)\in\Pi(f(t)^{\otimes N};F_N(t))\,,\quad\hbox{ for each }t\in\bR\,,
$$
and
$$
T_\si\pi_N(t)=\pi_N(t)\quad\hbox{ for each }t\in\bR\hbox{ and each }\si\in\fS_N\,.
$$
\end{Lem}

\begin{proof}
Let $\phi\equiv\phi(X_N,\Xi_N)$ and $\psi\equiv\psi(Y_N,H_N)\in C^\infty_c((\bR^d\times\bR^d)^N)$ be two test functions. Then
$$
\ba
\frac{d}{dt}\iint\phi(X_N,\Xi_N)\pi_{N,1}(t,dX_Nd\Xi_N)
\\
=\frac{d}{dt}\iint\iint\phi(X_N,\Xi_N)\pi_N(t,dX_Nd\Xi_NdY_NdH_N)
\\
=\iint\iint\{\bH_N^{\rho[f(t)]},\phi\}_N(X_N,\Xi_N)\pi_{N}(t,dX_Nd\Xi_NdY_NdH_N)
\\
=\iint\iint\{\bH_N^{\rho[f(t)]},\phi\}_N(X_N,\Xi_N)\pi_{N,1}(t,dX_Nd\Xi_N)
\ea
$$
since
$$
\{\cH_N(Y_N,H_N),\phi(X_N,\Xi_N)\}_{2N}=0\,.
$$
The penultimate chain of equalities shows that the first marginal $\pi_{N,1}$ of $\pi_N$ corresponding to the phase space variables $(X_N,\Xi_N)$ is a solution of the equation
$$
\d_t\pi_{N,1}+\{\bH_N^{\rho[f(t)]},\pi_{N,1}\}_N=0\,.
$$
On the other hand, an elementary computation shows that the solution $f$ of (\ref{Vlas}) satisfies
$$
\d_tf^{\otimes N}+\{\bH_N^{\rho[f(t)]},f^{\otimes N}\}_N=0\,.
$$
Since $\pi_{N,1}$ and $f^{\otimes N}$ are solutions of the same Liouville equation and 
$$
\pi_{N,1}(0)=(f^{in})^{\otimes N}=f(0)^{\otimes N}\,,
$$
we conclude from the uniqueness of the solution of the Cauchy problem for a transport equation with Lipschitz continuous coefficients that $\pi_{N,1}(t)=f(t)^{\otimes N}$ for all $t\ge 0$.

Similarly
$$
\ba
\frac{d}{dt}\iint\psi(Y_N,H_N)\pi_{N,2}(t,dY_NdH_N)&
\\
=\frac{d}{dt}\iint\iint\psi(Y_N,H_N)\pi_N(t,dX_Nd\Xi_NdY_NdH_N)&
\\
=\iint\iint\{\cH_N,\psi\}_N(Y_N,H_N)\pi_{N}(t,dX_Nd\Xi_NdY_NdH_N)&
\\
=\iint\iint\{\cH_N,\psi\}_N(Y_N,H_N)\pi_{N,2}(t,dY_NdH_N)&\,,
\ea
$$
since
$$
\{\bH_N^{\rho[f(t)]}(X_N,\Xi_N),\psi(Y_N,H_N)\}_{2N}=0\,.
$$
This shows that the second marginal $\pi_{N,2}$, corresponding to the phase space variables $(Y_N,H_N)$, is a solution to the same Liouville equation (\ref{NLiouv}) as $F_N$. Since
$$
\pi_{N,2}(0)=(f^{in})^{\otimes N}=F_N(0)\,,
$$
we conclude that $\pi_{N,2}(t)=F_N(t)$ for each $t\ge 0$, by uniqueness of the solution of (\ref{NLiouv}).

Finally the time-dependent Hamiltonian
$$
\bh_f:\,(X_N,\Xi_N,Y_N,H_N)\mapsto\bH_N^{\rho[f(t)]}(X_N,\Xi_N)+\cH_N(Y_N,H_N)
$$
satisfies
$$
\bh_f\circ T_\si=\bh_f\quad\hbox{ for all }\si\in\fS_N\,.
$$
Hence
$$
\d_t(\pi_N-T_\si\#\pi_N)(t)+\{\bh_f,(\pi_N-T_\si\#\pi_N)(t)\}_{2N}=0\,.
$$
Since $(\pi_N-T_\si\#\pi_N)(0)$ by (\ref{SympiN0}), we conclude that
$$
\pi_N(t)=T_\si\#\pi_N(t)
$$
for all $t\ge 0$, by uniqueness of the solution of the Cauchy problem for the Liouville equation with Hamiltonian $\bh_f$.
\end{proof}

\subsection{The Eulerian variant of the Dobrushin estimate}


Set
$$
D^p_N(t):=\int\frac1N\sum_{j=1}^N(|x_j-y_j|^p+|\xi_j-\eta_j|^p)\pi_N(t,dX_Nd\Xi_NdY_NdH_N)\,.
$$
We recall that $p\ge 1$. Then
$$
\ba
\frac{dD^p_N}{dt}=&\int\left\{\bH_N^{\rho[f(t)]}(X_N,\Xi_N),\frac1N\sum_{j=1}^N(|x_j-y_j|^p+|\xi_j-\eta_j|^p)\right\}_{2N}d\pi_N(t)
\\
&+\int\left\{\cH_N(Y_N,H_N),\frac1N\sum_{j=1}^N(|x_j-y_j|^p+|\xi_j-\eta_j|^p)\right\}_{2N}d\pi_N(t)\,.
\ea
$$

First
$$
\ba
\left\{\bH_N^{\rho[f(t)]}(X_N,\Xi_N),\frac1N\sum_{j=1}^N(|x_j-y_j|^p+|\xi_j-\eta_j|^p)\right\}_{2N}
\\
=
\frac1N\sum_{j=1}^N\{\tfrac12|\xi_j|^2,|x_j-y_j|^p\}_{2N}
+
\frac1N\sum_{j=1}^N\{V_{\rho[f]}(x_j),|\xi_j-\eta_j|^p\}_{2N}
\\
=\frac{p}N\sum_{j=1}^N\xi_j\cdot(x_j-y_j)|x_j-y_j|^{p-2}
-
\frac{p}N\sum_{j=1}^N\grad V_{\rho[f]}(x_j)\cdot(\xi_j-\eta_j)|\xi_j-\eta_j|^{p-2}\,,
\ea
$$
while
$$
\ba
\left\{\cH_N(Y_N,H_N),\frac1N\sum_{j=1}^N(|x_j-y_j|^p+|\xi_j-\eta_j|^p)\right\}_{2N}
\\
=
\frac1N\sum_{j=1}^N\{\tfrac12|\eta_j|^2,|x_j-y_j|^p\}_{2N}
+
\frac1{N^2}\sum_{j,k=1}^N\{V(y_j-y_k),|\xi_j-\eta_j|^p\}_{2N}
\\
=\frac{p}N\sum_{j=1}^N\eta_j\cdot(y_j-x_j)|y_j-x_j|^{p-2}
-
\frac{p}{N^2}\sum_{j,k=1}^N\grad V(y_j-y_k)\cdot(\eta_j-\xi_j)|\eta_j-\xi_j|^{p-2}\,.
\ea
$$
Therefore
$$
\ba
\frac{dD^p_N}{dt}=\frac{p}N\sum_{j=1}^N\int(\xi_j-\eta_j)\cdot(x_j-y_j)|x_j-y_j|^{p-2}d\pi_N&
\\
-\frac{p}N\sum_{j=1}^N\int(\xi_j-\eta_j)\cdot\left(\grad V_{\rho[f]}(x_j)-\frac1{N}\sum_{k=1}^N\grad V(y_j-y_k)\right)|\xi_j-\eta_j|^{p-2}d\pi_N&\,.
\ea
$$
At this point we use Young's inequality in the form
$$
pab^{p-1}\le a^p+(p-1)b^p\le\max(1,p-1)(a^p+b^p)
$$
for each $a,b>0$ and each $p\ge 1$. Denoting
$$
K_p:={\max(1,p-1)}\,,
$$
one has
$$
\ba
\frac{dD^p_N}{dt}&\le\frac{K_p}N\sum_{j=1}^N\int(|\xi_j-\eta_j|^p+|x_j-y_j|^p)d\pi_N+\frac{K_p}N\sum_{j=1}^N\int|\xi_j-\eta_j|^pd\pi_N
\\
&+\frac{K_p}N\sum_{j=1}^N\int\left|\grad V_{\rho[f]}(x_j)-\frac1{N}\sum_{k=1}^N\grad V(y_j-y_k)\right|^pd\pi_N
\\
&\le2K_pD^p_N+\frac{K_p}N\sum_{j=1}^N\int\left|\grad V_{\rho[f]}(x_j)-\frac1{N}\sum_{k=1}^N\grad V(y_j-y_k)\right|^pd\pi_N\,.
\ea
$$

Let us decompose this last term as follows
$$
\ba
\grad V_{\rho[f]}(x_j)-\frac1{N}\sum_{k=1}^N\grad V(y_j-y_k)
&=
\grad V_{\rho[f]}(x_j)-\frac1{N}\sum_{k=1}^N\grad V(x_j-x_k)
\\
&+
\frac1{N}\sum_{k=1}^N\left(\grad V(x_j-x_k)-\grad V(y_j-y_k)\right)\,,
\ea
$$
so that, by convexity of the function $z\mapsto z^p$ on $(0,\infty)$ for $p\ge 1$,
$$
\ba
\left|\grad V_{\rho[f]}(x_j)-\frac1{N}\sum_{k=1}^N\grad V(y_j-y_k)\right|^p&
\\
\le
2^{p-1}\left|\grad V_{\rho[f]}(x_j)-\frac1{N}\sum_{k=1}^N\grad V(x_j-x_k)\right|^p&
\\
+2^{p-1}
\left|\frac1{N}\sum_{k=1}^N(\grad V(x_j-x_k)-\grad V(y_j-y_k))\right|^p&\,.
\ea
$$
Then, by the same convexity argument as above
$$
\ba
\left|\frac1{N}\sum_{k=1}^N(\grad V(x_j-x_k)-\grad V(y_j-y_k))\right|^p&
\\
\le\frac1{N}\sum_{k=1}^N\left|\grad V(x_j-x_k)-\grad V(y_j-y_k)\right|^p&
\\
\le\frac{\Lip(\grad V)^p}{N}\sum_{k=1}^N|(x_j-x_k)-(y_j-y_k)|^p&
\\
\le\frac{2^{p-1}\Lip(\grad V)^p}{N}\sum_{k=1}^N(|x_j-y_j|^p+|x_k-y_k|^p)&\,.
\ea
$$
Hence
$$
\ba
\frac{dD^p_N}{dt}\le 2K_pD^p_N+\frac{2^{p-1}K_p\Lip(\grad V)^p}{N^2}\sum_{j,k=1}^N\int(|x_j-y_j|^p+|x_k-y_k|^p)d\pi_N&
\\
+\frac{2^{p-1}K_p}N\sum_{j=1}^N\int\left|\grad V_{\rho[f]}(x_j)-\frac1{N}\sum_{k=1}^N\grad V(x_j-x_k)\right|^pd\pi_N&\,.
\ea
$$
Since
$$
\frac1{N^2}\sum_{j,k=1}^N\int(|x_j-y_j|^p+|x_k-y_k|^p)d\pi_N=\frac2{N}\sum_{l=1}^N\int|x_l-y_l|^pd\pi_N\le 2D^p_N\,,
$$
the inequality above can be recast as
$$
\ba
\frac{dD^p_N}{dt}\le\L_pD^p_N+\frac{2^{p-1}K_p}N\sum_{j=1}^N\int\left|\grad V_{\rho[f]}(x_j)-\frac1{N}\sum_{k=1}^N\grad V(x_j-x_k)\right|^pd\pi_N&\,,
\ea
$$
with
$$
\L_p:=2K_p(1+2^{p-1}\Lip(\grad V)^p)\,.
$$

\subsection{Controlling the consistency error}

Let us examine the last term on the right hand side of this inequality. Since $\pi_N(t)\in\Pi(f(t)^{\otimes N},F_N(t))$ by Lemma \ref{L-DynCoupl}
$$
\ba
\int&\left|\grad V_{\rho[f]}(x_j)-\frac1{N}\sum_{k=1}^N\grad V(x_j-x_k)\right|^pd\pi_N
\\
&=
\int\left|\grad V_{\rho[f]}(x_j)-\frac1{N}\sum_{k=1}^N\grad V(x_j-x_k)\right|^p\rho[f]^{\otimes N}(t,dX_N)\,.
\ea
$$
Observe that this term involves only the factorized distribution $f^{\otimes N}$. This is the ``consistency'' error in the sense of numerical analysis. In other words, it measures by how much $f^{\otimes N}$ fails to be an exact solution of the
$N$-body Liouville equation (\ref{NLiouv}). Eventually we arrive at the inequality
\be\lb{Gronw<}
\ba
\frac{dD^p_N}{dt}\le\L_pD^p_N&
\\
+\frac{2^{p-1}K_p}N\sum_{j=1}^N\int\left|\grad V_{\rho[f]}(x_j)-\frac1{N}\sum_{k=1}^N\grad V(x_j-x_k)\right|^p\rho[f]^{\otimes N}(t,dX_N)&\,.
\ea
\ee

The last term on the right hand side of (\ref{Gronw<}) is mastered by the following inequality\footnote{We thank one of the referees who suggested using this argument instead of the error estimates in terms of Monge-Kantovich distances
for the law of large numbers obtained by Fournier and Guillin \cite{FourGuil}.}.

\begin{Lem}\lb{L-CombIneq}
Let $F$ be a bounded vector field on $\bR^d$, and $\rho$ be a probability density on $\bR^d$. For each $p>0$ and each $j=1,\ldots,N$, one has
$$
\int\left|F\star\rho(x_j)-\frac1N\sum_{k=1}^NF(x_j-x_k)\right|^{p}\prod_{m=1}^N\rho(x_m)dx_m
\le
\frac{2[p/2]+2}{N^{\min(p/2,1)}}(2\|F\|_{L^\infty})^{p}\,.
$$
\end{Lem}

The proof of Lemma \ref{L-CombIneq} is deferred until Appendix \ref{A-CombIneq}. Inserting the bound in Lemma \ref{L-CombIneq} in the right hand side of the differential inequality (\ref{Gronw<}) leads to
$$
\ba
\frac{dD^{p}_N}{dt}\le\L_{p}D^{p}_N+2^{2p}K_{p}\frac{[p/2]+1}{N^{\min(p/2,1)}}\|\grad V\|_{L^\infty}^{p}\,.
\ea
$$
Gronwall's inequality implies that
\be\lb{IneqD}
D^{p}_N(t)\le D^{p}_N(0)e^{\L_{p}t}+2^{2p}K_{p}\|\grad V\|_{L^\infty}^{p}\frac{[p/2]+1}{N^{\min(p/2,1)}}\frac{e^{\L_{p}t}-1}{\L_{p}}\,.
\ee

\subsection{End of the proof of Theorem \ref{T-MFC}}

By Lemma \ref{L-DynCoupl}, 
$$
T_\si\#\pi_N(t)=\pi_N(t)
$$
for each $t\ge 0$ and each permutation $\si\in\fS_N$. Thus, for each $k=1,\ldots,N$, one has
$$
\ba
\int(|x_k-y_k|^p+|\xi_k-\eta_k|^p)\pi_N(t,dX_Nd\Xi_NdY_NdH_N)&
\\
=
\int(|x_1-y_1|^p+|\xi_1-\eta_1|^p)\pi_N(t,dX_Nd\Xi_NdY_NdH_N)&\,,
\ea
$$
so that
$$
\ba
D^p_N(t)=\int(|x_j-y_j|^p+|\xi_j-\eta_j|^p)\pi_N(t,dX_Nd\Xi_NdY_NdH_N)
\\
=\frac1n\int\sum_{j=1}^n(|x_j-y_j|^p+|\xi_j-\eta_j|^p)\pi_N(t,dX_Nd\Xi_NdY_NdH_N)
\ea
$$

Denoting 
\be\lb{DefXNk}
X_N^k:=(x_k,\ldots,x_N)\,,\quad\Xi_N^k:=(\xi_k,\ldots,\xi_N)
\ee
for each $k=1,\ldots N$, we set
$$
\pi^\mathbf{n}_N(t,dX_nd\Xi_ndY_ndH_n):=\int\pi_N(t,dX_N^nd\Xi_N^ndY_N^ndH_N^n)\,.
$$
By Fubini's theorem, for each $\phi\in C_b((\bR^{2d})^n)$, one has
$$
\ba
\int\phi(X_n,\Xi_n)\pi^\mathbf{n}_N(t,dX_nd\Xi_ndY_ndH_n)&=\int\phi(X_n,\Xi_n)\pi_N(t,dX_Nd\Xi_NdY_NdH_N)
\\
&=\int\phi(X_n,\Xi_n)f(t)^{\otimes N}(dX_Nd\Xi_N)
\\
&=\int\phi(X_n,\Xi_n)f(t)^{\otimes n}(dX_nd\Xi_n)\,,
\ea
$$
while
$$
\ba
\int\phi(Y_n,H_n)\pi^\mathbf{n}_N(t,dX_nd\Xi_ndY_ndH_n)&=\int\phi(Y_n,H_n)\pi_N(t,dX_Nd\Xi_NdY_NdH_N)
\\
&=\int\phi(Y_n,H_n)F_N(t,dY_NdH_N)
\\
&=\int\phi(Y_n,H_n)F_N^\mathbf{n}(t,dY_ndH_n)\,,
\ea
$$
so that
$$
\pi^\mathbf{n}_N(t)\in\Pi(f(t)^{\otimes n},F^\mathbf{n}_N(t))\quad\hbox{ for each }t\ge 0\,.
$$

Therefore
\be\lb{MinDNp}
\ba
D^p_N(t)&=\frac1n\int\sum_{j=1}^n(|x_j-y_j|^p+|\xi_j-\eta_j|^p)\pi^\mathbf{n}_N(t,dX_nd\Xi_ndY_ndH_n)
\\
&\ge\frac1n\MKp(f(t)^{\otimes n},F_N^\mathbf{n}(t))^p\,.
\ea
\ee

Thus, inequality (\ref{IneqD}) implies that
\be\lb{IneqMKp}
\ba
\frac1n\MKp(f(t)^{\otimes n},F_N^\mathbf{n}(t))^p\le &D^{p}_N(0)e^{\L_{p}t}
\\
&+2^{2p}K_{p}\|\grad V\|_{L^\infty}^{p}\frac{[p/2]+1}{N^{\min(p/2,1)}}\frac{e^{\L_{p}t}-1}{\L_{p}}
\ea
\ee
for each $t\ge 0$, each $N\ge n\ge 1$ and each $\pi_N^{in}\in\Pi((f^{in})^{\otimes N},F_N^{in})$ satisfying
$$
T_\si\pi_N^{in}=\pi_N^{in}\quad\hbox{  for all }\si\in\fS_N\,.
$$

Finally, choose the initial coupling of the form
$$
\pi_N^{in}:=(f^{in})^{\otimes N}(dX_Nd\Xi_N)\de_{(X_N,\Xi_N)}(Y_N,H_N)\,,
$$
i.e.
$$
\pi_N:=\bD\#(f^{in})^{\otimes N}\,,\quad\hbox{ where }\bD:\,(X_N,\Xi_N)\mapsto(X_N,\Xi_N,X_N,\Xi_N)\,.
$$
Then
\be\lb{DNp0}
D^p_N(0)=\frac1N\sum_{j=1}^N\int(|x_k-y_k|^p+|\xi_k-\eta_k|^p)\pi_N^{in}(dX_Nd\Xi_NdY_NdH_N)=0\,,
\ee
since $\pi_N^{in}$ is supported in the diagonal
$$
\{(X_N,\Xi_N,Y_N,H_N)\hbox{ s.t. }X_N=Y_N\hbox{ et }\Xi_N=H_N\}\,.
$$
Inserting this last piece of information in (\ref{IneqMKp}) leads to the inequality in the statement of Theorem \ref{T-MFC}. 
\end{proof}

\begin{Rmk}\lb{R-dMKp}
In the argument above, we have not checked that $f(t)^{\otimes n}$ and $F_N^\mathbf{n}(t)$ belong to $\cP_p((\bR^d)^n)$ for all $1\le n\le N$ and all $t\ge 0$. This verification is needed on principle, because the Monge-Kantorovich
distance of exponent $p$ is defined on the set of probability measures with finite moments of order $p$. For each $p\ge 2$, set
$$
M_p(t):=\iint(|x|^p+|\xi|^p)f(t,dxd\xi)\,,\qquad t\ge 0\,,
$$
where $f$ is the solution of (\ref{Vlas}). An easy argument based on Young's inequality and the mean value inequality for $\grad V$ shows that
$$
M_p(t)\le M_p(0)e^{(p-1)(1+2\Lip(\grad V))t}\,,\qquad t\ge 0\,.
$$
In particular $f(t)\in\cP_p(\bR^d\times\bR^d)$ for all $t\ge 0$ provided that $f^{in}\in\cP_p(\bR^d\times\bR^d)$. Therefore $f(t)^{\otimes n}\in\cP_p((\bR^d\times\bR^d)^n)$ for each $n\ge 1$, and the inequality in Theorem \ref{T-MFC}
implies that $F(t)_N^\mathbf{n}\in\cP_p((\bR^d\times\bR^d))$ for all $t\ge 0$ and each $n=1,\ldots,N$, since $\MKp(f(t)^{\otimes n},F(t)_N^\mathbf{n})<\infty$.
\end{Rmk}


\section{Proof of Theorem \ref{T-MKeps}: Properties of $MK^\eps_2$}\lb{S-MKeps}


\subsection{The general lower bound (\ref{LBMKeps})}


The lower bound (\ref{LBMKeps}) can be viewed as a variant of the uncertainty principle. More precisely
$$
Q^*Q+P^*P=(Q+iP)^*(Q+iP)+i(P^*Q-Q^*P)\ge i(P^*Q-Q^*P)\,.
$$
On the other hand,
$$
(P^*Q-Q^*P)=-i\eps\left(\Div_{x_1}(x_1)+\Div_{x_2}(x_2)\right)=-2id\eps I_{\fH_2}\,,
$$
so that
$$
Q^*Q+P^*P\ge 2d\eps I_{\fH_2}\,.
$$
Therefore, for each $R\in\cD(\fH_2)$, one has
$$
\ba
\Tr_{\fH_2}((Q^*Q+P^*P)R)&=\Tr_{\fH_2}(R^{1/2}(Q^*Q+P^*P)R^{1/2})
\\
&\ge 2d\eps\Tr_{\fH_2}(R)=2d\eps\,.
\ea
$$

\subsection{Upper bound for T\"oplitz operators}


The goal of this section is to prove statement (1) in Theorem \ref{T-MKeps}.

In the course of the proof, we shall need the following intermediate result. Everywhere in this section, the identification $\bR^{2d}\ni(q,p)\mapsto q+ip\in\bC^d$ is implicitly assumed. 

\begin{Lem}[Construction of quantum couplings]\lb{L-QCoupl}
Let $\eps>0$ and let $\rho^\eps_1$ and $\rho^\eps_2$ be T\"oplitz operators at scale $\eps$ on $L^2(\bR^d)$ with symbols $(2\pi\eps)^d\mu_1$ and $(2\pi\eps)^d\mu_2$, where $\mu_1,\mu_2\in\cP(\bR^{2d})$. For each coupling 
$\mu\in\Pi(\mu_1,\mu_2)$, the T\"oplitz operator $\Op^T_\eps((2\pi\eps)^{2d}\mu)$ belongs to $\cQ(\rho^\eps_1,\rho^\eps_2)$.
\end{Lem}

\begin{proof}
Set $\rho^\eps_j=\Op^T_\eps((2\pi\eps)^d\mu_j)$ for $j=1,2$. We recall that $\fH=L^2(\bR^d)$. 

Since $\mu$ is a probability measure on $\bC^{2d}$, the operator $\Op^T_\eps((2\pi\eps)^{2d}\mu)$ belongs to $\cD(\fH\otimes\fH)$ by (\ref{SAToep}). Next, one has
$$
\Op_\eps^T((2\pi\eps)^{2d}\mu)=\int_{\bC^{2d}}|(z_1,z_2),\eps\ra\la(z_1,z_2),\eps|\mu(dz_1dz_2)\,,
$$
where the wave function $|(z_1,z_2),\eps\ra$ is defined in formula (\ref{CohSt}) in Appendix \ref{A-Toeplitz}. In this section, we use the bra-ket notation, also recalled in Appendix \ref{A-Toeplitz}. Observe that
$$
|(z_1,z_2),\eps\ra\la(z_1,z_2),\eps|=|z_1,\eps\ra\la z_1,\eps|\otimes|z_2,\eps\ra\la z_2,\eps|\,,
$$
so that
$$
\ba
\Tr_{\fH\otimes\fH}((A\otimes I_\fH)|(z_1,z_2),\eps\ra\la(z_1,z_2),\eps|)&
\\
=\Tr_\fH(A|z_1,\eps\ra\la z_1,\eps|)\Tr_\fH(|z_2,\eps\ra\la z_2,\eps|)&
\\
=\Tr_\fH(A|z_1,\eps\ra\la z_1,\eps|)\la z_2,\eps|z_2,\eps\ra&
\\
=\Tr_\fH(A|z_1,\eps\ra\la z_1,\eps|)&\,,
\ea
$$
for each $A\in\cL(\fH)$. Therefore, 
$$
\ba
\Tr_{\fH\otimes\fH}((A\otimes I_{\fH})\Op_\eps^T((2\pi\eps)^{2d}\mu))&
\\
=\int_{\bC^{2d}}\Tr_{\fH\otimes\fH}((A\otimes I_{\fH})|(z_1,z_2),\eps\ra\la(z_1,z_2),\eps|)\mu(dz_1dz_2)&
\\
=\int_{\bC^{2d}}\Tr_{\fH}(A|z_1,\eps\ra\la z_1,\eps|)\mu(dz_1dz_2)&
\\
=\int_{\bC^{d}}\Tr_{\fH}(A|z_1,\eps\ra\la z_1,\eps|)\mu_1(dz_1)&
\\
=\Tr_{\fH}(A\Op_\eps^T((2\pi\eps)^{d}\mu_1))&\,.
\ea
$$
Likewise, for each $A\in\cL(\fH)$,
$$
\Tr_{\fH\otimes\fH}((I_\fH\otimes A)\Op^T_\eps((2\pi\eps)^{2d}\mu))=\Tr_{\fH}(A\Op^T_\eps((2\pi\eps)^d\mu_2))=\Tr_{\fH}(A\rho^\eps_2)\,.
$$
The conclusion immediately follows.
\end{proof}

\begin{proof}[Proof of Theorem \ref{T-MKeps} (1)]
For each $\mu\in\Pi(\mu_1,\mu_2)$, by Lemma \ref{L-QCoupl} and the definition of $MK^\eps_2$ , one has
$$
MK^\eps_2(\rho_1,\rho_2)^2\le\Tr_{\fH\otimes\fH}((Q^*Q+P^*P)\Op^T_\eps((2\pi\eps)^{2d}\mu))\,,
$$
where we recall that
\be\lb{DefQP}
\ba
Q\psi(x_1,x_2)=(x_1-x_2)\psi(x_1,x_2)\,,
\\
P\psi(x_1,x_2)=-i\eps(\grad_{x_1}-\grad_{x_2})\psi(x_1,x_2)\,,
\ea
\ee
so that
$$
Q^*Q\psi(x_1,x_2)=|x_1-x_2|^2\psi(x_1,x_2)\,,
$$
and
$$
P^*P\psi(x_1,x_2)=-\eps^2(\Div_{x_1}-\Div_{x_2})(\grad_{x_1}-\grad_{x_2})\psi(x_1,x_2)\,.
$$
In other words
$$
Q^*Q+P^*P=\Op^T_\eps(|q_1-q_2|^2+|p_1-p_2|^2)-2d\eps I_{\fH_2}
$$
--- see formula (\ref{ToepQuad}) in Appendix \ref{A-Toeplitz}, with $f(y_1,y_2)=|y_1-y_2|^2$, and $\Dlt f=4d$. By formula (\ref{MK2}) in Appendix \ref{A-Toeplitz}, one has
\be\lb{DToepl=}
\ba
\Tr_{\fH\otimes\fH}((Q^*Q+P^*P)\Op^T_\eps((2\pi\eps)^{2d}\mu))&
\\
=\int_{(\bR^d)^4}(|x_1-x_2|^2+|\xi_1-\xi_2|^2)\mu(dx_1d\xi_1dx_2d\xi_2)+2d\eps&\,.
\ea
\ee

Therefore, for each $\mu\in\Pi(\mu_1,\mu_2)$, one has
$$
MK^\eps_2(\rho_1,\rho_2)^2\le\int_{(\bR^d)^4}(|x_1-x_2|^2+|\xi_1-\xi_2|^2)\mu(dx_1d\xi_1dx_2d\xi_2)+2d\eps\,.
$$
Observing that the left hand side of this inequality is independent of $\mu\in\Pi(\mu_1,\mu_2)$, we conclude that
$$
\ba
MK^\eps_2(\rho_1,\rho_2)^2&\le\inf_{\pi\in\Pi(\mu_1,\mu_2)}\int_{(\bR^d)^4}(|x_1-x_2|^2+|\xi_1-\xi_2|^2)\mu(dx_1d\xi_1dx_2d\xi_2)+2d\eps
\\
&=\MKd(\mu_1,\mu_2)^2+2d\eps
\ea
$$
which is the sought inequality.
\end{proof}

\subsection{Asymptotic lower bound for $MK^\eps_2$}


The core of the argument leading to the lower bound in Theorem \ref{T-MKeps} (2) combines Kantorovich duality with the convergence of Husimi functions to Wigner measures.

\begin{proof}[Proof of Theorem \ref{T-MKeps} (2)]
Let $a,b\in C_b(\bR^d)$ satisfy
\be\lb{Ineqab}
a(x_1,\xi_1)+b(x_2,\xi_2)\le|x_1-x_2|^2+|\xi_1-\xi_2|^2\quad\hbox{ for all }x_1,x_2,\xi_1,\xi_2\in\bR^d\,.
\ee
Hence\footnote{Denoting $a\otimes 1$, $1\otimes b$ and $c$ the functions $(x_1,\xi_1,x_2,\xi_2)\mapsto a(x_1,\xi_1)$, $(x_1,\xi_1,x_2,\xi_2)\mapsto b(x_2,\xi_2)$ and $(x_1,\xi_1,x_2,\xi_2)\mapsto|x_1-x_2|^2+|\xi_1-\xi_2|^2$ respectively.}
$$
\ba
(\Op^T_\eps(a)\otimes I_\fH+I_\fH\otimes\Op^T_\eps(b))&=\Op^T_\eps(a\otimes 1+1\otimes b)
\\
&\le\Op^T_\eps(c)=Q^*Q+P^*P+2d\eps I_{\fH_2}\,.
\ea
$$
Thus, for each $R^\eps\in\cQ(\rho^\eps_1,\rho^\eps_2)$, one has
$$
\ba
\Tr_{\fH\otimes\fH}((Q^*Q+P^*P)R^\eps)&
\\
\ge\Tr_{\fH\otimes\fH}((\Op^T_\eps(a)\otimes I_\fH+I_\fH\otimes\Op^T_\eps(b))R^\eps)-2d\eps&
\\
=\Tr_{\fH}(\Op^T_\eps(a)\rho^\eps_1)+\Tr_{\fH}(\Op^T_\eps(b)\rho^\eps_2)-2d\eps&\,.
\ea
$$
Taking the inf of the left hand side as $R^\eps$ runs through $\cQ(\rho^\eps_1,\rho^\eps_2)$, one arrives at the inequality
\be\lb{1LowBdMKeps}
MK^\eps_2(\rho^\eps_1,\rho^\eps_2)^2\ge\Tr_{\fH}(\Op^T_\eps(a)\rho^\eps_1)+\Tr_{\fH}(\Op^T_\eps(b)\rho^\eps_2)-2d\eps\,.
\ee
Next the traces on the right hand side are expressed in terms of the Husimi functions of $\rho^\eps_1$ and $\rho^\eps_2$ by formula (\ref{TrToep}):
$$
\ba
\Tr_{\fH}(\Op^T_\eps(a)\rho^\eps_1)=\iint a(x_1,\xi_1)\tilde W_\eps[\rho^\eps_1](x_1,\xi_1)dx_1d\xi_1\,,
\\
\Tr_{\fH}(\Op^T_\eps(b)\rho^\eps_2)=\iint b(x_2,\xi_2)\tilde W_\eps[\rho^\eps_2](x_2,\xi_2)dx_2d\xi_2\,.
\ea
$$
Hence
$$
\ba
MK^\eps_2(\rho^\eps_1,\rho^\eps_2)^2&\ge\iint a(x_1,\xi_1)\tilde W_\eps[\rho^\eps_1](x_1,\xi_1)dx_1d\xi_1
\\
&+\iint b(x_2,\xi_2)\tilde W_\eps[\rho^\eps_2](x_2,\xi_2)dx_2d\xi_2-2d\eps\,.
\ea
$$
Taking the sup of both sides of this equality over all $a,b\in C_b(\bR^d\times\bR^d)$ satisfying (\ref{Ineqab}) shows that
$$
\ba
MK^\eps_2(\rho^\eps_1,\rho^\eps_2)^2&
\\
\ge\sup_{a\otimes 1+1\otimes b\le c\atop a,b\in C_b(\bR^d\times\bR^d)}\left(\iint a\tilde W_\eps[\rho^\eps_1]dx_1d\xi_1+\iint b\tilde W_\eps[\rho^\eps_2]dx_2d\xi_2\right)-2d\eps&
\\
=\MKd(\tilde W_\eps[\rho^\eps_1],\tilde W_\eps[\rho^\eps_2])^2-2d\eps&\,,
\ea
$$
where the last equality follows from Kantorovich duality (Theorem 1 in chapter 1 of \cite{VillaniTOT}). This gives the first inequality in Theorem \ref{T-MKeps} (2).

Since the sequence of Wigner transforms of the density matrices $\rho^\eps_j$ satisfies 
$$
W_\eps[\rho^\eps_j]\to w_j\quad\hbox{ in }\cS'(\bR^d\times\bR^d)\hbox{  as }\eps\to 0
$$ 
for $j=1,2$, one has
$$
\tilde W_\eps[\rho^\eps_j]\to w_j\hbox{ weakly in the sense of measures on }\bR^d\times\bR^d
$$
for $j=1,2$, by Theorem III.1 (1) in \cite{PLLTP}. By the first inequality in Theorem \ref{T-MKeps} (2) already established above and Remark 6.12 in \cite{VillaniTOT2},
$$
\varliminf_{\eps\to 0}MK^\eps_2(\rho^\eps_1,\rho^\eps_2)^2\ge\varliminf_{\eps\to 0}\MKd(W_\eps[\rho^\eps_1],W_\eps[\rho^\eps_2])^2\ge\MKd(w_1,w_2)^2\,.
$$
This concludes the proof of Theorem \ref{T-MKeps}.
\end{proof}


\section{Proof of Theorem \ref{T-UQMF}}\label{S-UQMF}


The quantum $N$-body Hamiltonian is
$$
\cH_{\eps,N}:=\sum_{j=1}^N-\tfrac12\eps^2\Dlt_j+\frac1{2N}\sum_{1\le j,k\le N}V_{jk}\,,
$$
an unbounded self-adjoint operator on $\fH_N:=\fH^{\otimes N}=L^2((\bR^d)^N)$. We denote by $\Dlt_j$ the Laplacian acting on the variable $x_j$, and by $V_{jk}$ the multiplication by $V(x_j-x_k)$. The $N$-body von Neumann equation 
for the density matrix is
\be\lb{LiouvilleQ}
i\eps\d_t\rho_{\eps,N}:=[\cH_{\eps,N},\rho_{\eps,N}]\,,\qquad\rho_{\eps,N}\rstr_{t=0}=\rho_{\eps,N}^{in}\,,
\ee
with initial data $\rho_{\eps,N}^{in}\in\cD(\fH_N)$.

On the other hand, for each $\rho\in\cD(\fH)$, we consider the mean field Hamiltonian
$$
\bH^\rho_\eps:=-\tfrac12\eps^2\Dlt+V_\rho\,,
$$
where $V_\rho$ designates the operator defined on $\fH$ as the multiplication by the function
$$
x\mapsto\int V(x-x')\rho(x',x')dx'=:V_\rho(x)\,.
$$
The Hartree equation for the density matrix $\rho_\eps$ is
\be\lb{HartreeRho}
i\eps\d_t\rho_\eps=[\bH^{\rho_\eps}_\eps,\rho_\eps]\,,\qquad\rho_\eps\rstr_{t=0}=\rho_\eps^{in}\,.
\ee
By a straightforward computation, the solution $\rho_\eps$ of the Hartree equation (\ref{HartreeRho}) satisfies
\be\lb{HartreeN}
i\eps\d_t\rho_\eps^{\otimes N}=[\bH_{\eps,N}^{\rho_\eps},\rho_\eps^{\otimes N}]\,,\qquad\rho_\eps^{\otimes N}\rstr_{t=0}=(\rho^{in})^{\otimes N}\,,
\ee
where
$$
\bH_{\eps,N}^\rho:=\bH_\eps^\rho\otimes I_{\fH_{N-1}}+I\otimes\bH_\eps^\rho\otimes I_{\fH_{N-2}}+\ldots+I_{\fH_{N-1}}\otimes\bH_\eps^\rho\,.
$$

\subsection{The quantum dynamics of couplings}


Let $R^{in}_{\eps,N}\in\cQ((\rho^{in}_\eps)^{\otimes N},\rho_{\eps,N}^{in})$, and define $t\mapsto R_{\eps,N}(t)\in\cD(\fH_{2N})$ to be the solution of the Cauchy problem
\be\lb{QCPCoupl}
i\eps\d_tR_{\eps,N}=[\bH_{\eps,N}^{\rho_\eps}\otimes I_{\fH_N}+ I_{\fH_N}\otimes\cH_{\eps,N},R_{\eps,N}]\,,\qquad R_{\eps,N}\rstr_{t=0}=R^{in}_{\eps,N}\,.
\ee
For each $\si\in\fS_N$, we denote by $\cT_\si$ the unitary operator on $\fH_{2N}$ defined by
$$
\cT_\si\Psi(x_1,\ldots,x_N,y_1,\ldots,y_N)=\Psi(x_{\si^{-1}(1)},\ldots,x_{\si^{-1}(N)},y_{\si^{-1}(1)},\ldots,y_{\si^{-1}(N)})
$$
for each $\Psi\in\fH_{2N}$. We shall henceforth assume that
\be\lb{SymReNin}
\cT_\si R^{in}_{\eps,N}\cT^*_\si=R^{in}_{\eps,N}\,,\quad\hbox{ for each }\si\in\fS_N\,.
\ee

\begin{Lem}\lb{L-RN}
For each $t\ge 0$, one has $R_{\eps,N}(t)\in\cQ((\rho_\eps(t))^{\otimes N},\rho_{\eps,N}(t))$. Moreover
\be\lb{SymReNt}
\cT_\si R_{\eps,N}(t)\cT^*_\si=R_{\eps,N}(t)\qquad\hbox{ for each }\si\in\fS_N\hbox{ and }t\ge 0\,.
\ee
\end{Lem}

\begin{proof}
By definition
$$
R_{\eps,N}(t)=U_{\eps,N}(t/\eps)R^{in}_{\eps,N}U_{\eps,N}(t/\eps)^*\,,
$$
where 
$$
\frac{d}{dt}U_{\eps,N}(t)=-i(\bH_{\eps,N}^{\rho_\eps}\otimes I_{\fH_N}+ I_{\fH_N}\otimes\cH_{\eps,N})U_{\eps,N}(t)\,,\quad U_{\eps,N}(0)=I_{\fH_{2N}}\,.
$$
Since $\bH_{\eps,N}^{\rho_\eps}\otimes I_{\fH_N}+ I_{\fH_N}\otimes\cH_{\eps,N}$ is self-adjoint, $U_{\eps,N}(t)$ is unitary for each $t\ge 0$. Therefore
$$
R_{\eps,N}(t)=R_{\eps,N}(t)^*\ge 0\quad\hbox{ and }\Tr_{\fH_{2N}}(R_{\eps,N}(t))=\Tr_{\fH_{2N}}(R^{in}_{\eps,N})=1
$$
for each $t\ge 0$.

The marginals of the density $R_{\eps,N}$ in the product $\fH_{2N}=\fH_N\otimes\fH_N$ are defined by analogy with (\ref{Margi12})-(\ref{DefMargi12}) in Definition \ref{D-DefQ}: 
\be\lb{QMargi12}
\left\{
\ba
R_{\eps,N,1}\in\cD(\fH_N)\quad\hbox{ and }\Tr_{\fH_N}(AR_{\eps,N,1})=\Tr_{\fH_{2N}}((A\otimes I_{\fH_N})R_{\eps,N})\,,
\\
R_{\eps,N,2}\in\cD(\fH_N)\quad\hbox{ and }\Tr_{\fH_N}(AR_{\eps,N,2})=\Tr_{\fH_{2N}}((I_{\fH_N}\otimes A)R_{\eps,N})\,,
\ea
\right.
\ee
for each $A\in\cL(\fH_N)$.

Let $A\in\cL(\fH_N)$ be such that 
$$
\left[\sum_{j=1}^N\Dlt_{x_j},A\right]\in\cL(\fH_N)\,;
$$
then
$$
\ba
i\eps\d_t\Tr_{\fH_N}(AR_{\eps,N,1})=i\eps\d_t\Tr_{\fH_{2N}}((A\otimes I_{\fH_N})R_{\eps,N})&
\\
=-\Tr_{\fH_{2N}}([\bH_{\eps,N}^{\rho_\eps}\otimes I_{\fH_N}+ I_{\fH_N}\otimes\cH_{\eps,N},(A\otimes I_{\fH_N})]R_{\eps,N})&
\\
=-\Tr_{\fH_{2N}}(([\bH_{\eps,N}^{\rho_\eps},A]\otimes I_{\fH_N})R_{\eps,N})&
\\
=-\Tr_{\fH_N}([\bH_{\eps,N}^{\rho_\eps},A]R_{\eps,N,1})&
\\
=\Tr_{\fH_N}(A[\bH_{\eps,N}^{\rho_\eps},R_{\eps,N,1}])&
\ea
$$
and 
$$
R_{\eps,N,1}\rstr_{t=0}=R_{\eps,N,1}^{in}=(\rho^{in})^{\otimes N}\,.
$$
By uniqueness of the solution of the Cauchy problem (\ref{HartreeN}), one concludes that
$$
R_{\eps,N,1}(t)=\rho_\eps(t)^{\otimes N}\,,\quad\hbox{ for each }t\ge 0\,.
$$
Similarly
$$
\ba
i\eps\d_t\Tr_{\fH_N}(AR_{\eps,N,2})=i\eps\d_t\Tr_{\fH_{2N}}((I_{\fH_N}\otimes A)R_{\eps,N})&
\\
=-\Tr_{\fH_{2N}}([\bH_{\eps,N}^{\rho_\eps}\otimes I_{\fH_N}+ I_{\fH_N}\otimes\cH_{\eps,N},(I_{\fH_N}\otimes A)]R_{\eps,N})&
\\
=-\Tr_{\fH_{2N}}((I_{\fH_N}\otimes[\cH_{\eps,N},A])R_{\eps,N})&
\\
=-\Tr_{\fH_N}([\cH_{\eps,N},A]R_{\eps,N,2})
\\
=\Tr_{\fH_N}(A[\cH_{\eps,N},R_{\eps,N,2}])&
\ea
$$
and 
$$
R_{\eps,N,2}\rstr_{t=0}=R_{\eps,N,2}^{in}=\rho_{\eps,N}^{in}\,.
$$
By uniqueness of the solution of the Cauchy problem for the von Neumann equation (\ref{LiouvilleQ}), one concludes that
$$
R_{\eps,N,2}(t)=\rho_{\eps,N}(t)\,,\quad\hbox{ for each }t\ge 0\,.
$$

Finally, let $\si\in\fS_N$;
$$
\ba
i\eps\d_t(\cT_\si R_{\eps,N}(t)\cT^*_\si)=\cT_\si[\bH_{\eps,N}^{\rho_\eps}\otimes I_{\fH_N}+ I_{\fH_N}\otimes\cH_{\eps,N},R_{\eps,N}]\cT^*_\si
\\
=
[\cT_\si(\bH_{\eps,N}^{\rho_\eps}\otimes I_{\fH_N}+ I_{\fH_N}\otimes\cH_{\eps,N})\cT^*_\si,\cT_\si R_{\eps,N}\cT^*_\si]
\\
=[\tau_\si\bH_{\eps,N}^{\rho_\eps}\tau^*_\si\otimes I_{\fH_N}+ I_{\fH_N}\otimes \tau_\si\cH_{\eps,N}\tau^*_\si,\cT_\si R_{\eps,N}\cT^*_\si]
\\
=[\bH_{\eps,N}^{\rho_\eps}\otimes I_{\fH_N}+ I_{\fH_N}\otimes\cH_{\eps,N},\cT_\si R_{\eps,N}\cT^*_\si]
\ea
$$
where $\tau_\si$ is the unitary operator on $\fH_N$ defined in (\ref{DefTauSi}), since one has obviously
$$
\tau_\si\bH_{\eps,N}^{\rho_\eps}\tau^*_\si=\bH_{\eps,N}^{\rho_\eps}\quad\hbox{ and }\quad\tau_\si\cH_{\eps,N}\tau^*_\si=\cH_{\eps,N}\,.
$$
On the other hand
$$
\cT_\si R_{\eps,N}(0)\cT^*_\si=\cT_\si R^{in}_{\eps,N}(0)\cT^*_\si=R^{in}_{\eps,N}=R_{\eps,N}(0)\,,
$$
so that, by uniqueness of the solution of the Cauchy problem (\ref{QCPCoupl}), 
$$
\cT_\si R_{\eps,N}(t)\cT^*_\si=R_{\eps,N}(t)\qquad\hbox{ for each }t\ge 0\hbox{   and each }\si\in\fS_N\,.
$$
\end{proof}

\subsection{The Dobrushin type estimate}


Set
$$
D_{\eps,N}(t):=\Tr\left(\frac1N\sum_{j=1}^N(Q^*_jQ_j+P^*_jP_j)R_{\eps,N}(t)\right)\,.
$$
We compute
$$
\ba
\frac{dD_{\eps,N}}{dt}&=-\frac1\eps\Tr\left(i\frac1N\sum_{j=1}^N(Q^*_jQ_j+P^*_jP_j)[\bH_{\eps,N}^{\rho_\eps}\otimes I_{\fH_N}+ I_{\fH_N}\otimes\cH_{\eps,N},R_{\eps,N}]\right)
\\
&=\frac1\eps\Tr\left(i\left[\bH_{\eps,N}^{\rho_\eps}\otimes I_{\fH_N}+ I_{\fH_N}\otimes\cH_{\eps,N},\frac1N\sum_{j=1}^N(Q^*_jQ_j+P^*_jP_j)\right]R_{\eps,N}\right)
\\
&=\frac1\eps\Tr\left(i\left[\bH_{\eps,N}^{\rho_\eps}\otimes I_{\fH_N}-\cH_{\eps,N}\otimes I_{\fH_N},\frac1N\sum_{j=1}^N(Q^*_jQ_j+P^*_jP_j)\right]R_{\eps,N}\right)
\\
&+\frac1\eps\Tr\left(i\left[\cH_{\eps,N}\otimes I_{\fH_N}+ I_{\fH_N}\otimes\cH_{\eps,N},\frac1N\sum_{j=1}^N(Q^*_jQ_j+P^*_jP_j)\right]R_{\eps,N}\right).
\ea
$$

Consider the second term on the right hand side of this equality. One has
$$
\ba
{}&\left[\cH_{\eps,N}\otimes I_{\fH_N}\!+\! I_{\fH_N}\otimes\cH_{\eps,N},(Q^*_jQ_j+P^*_jP_j)\right]
\\
&\qquad\qquad=
[-\tfrac12\eps^2(\Dlt_j\otimes I_{\fH}+I_{\fH}\otimes\Dlt_j),Q^*_jQ_j]_j
\\
&\qquad\qquad+
\frac1N\sum_{k=1}^N[V_{jk}\otimes I_{\fH}+I_{\fH}\otimes V_{jk},P_j^*P_j]_j
\ea
$$
where the index $j$ on the brackets in the right hand side indicate that the corresponding operators act on the variables $x_j$ et $y_j$. In other words
$$
A_j=I_\fH^{\otimes(j-1)}\otimes A\otimes I_\fH^{\otimes(N-j)}
$$
for each operator $A$ on $\fH$.

Obviously
$$
\ba
{}&\left[-\tfrac12\eps^2(\Dlt_j\otimes I_{\fH}+I_{\fH}\otimes\Dlt_j),Q^*_jQ_j\right]
\\
=&-\tfrac12\eps^2(\Div_{x_j}[\grad_{x_j},Q^*_jQ_j]\!+\![\grad_{x_j},Q^*_jQ_j]\cdot\grad_{x_j})
\\
&-\tfrac12\eps^2(\Div_{y_j}[\grad_{y_j},Q^*_jQ_j]\!+\![\grad_{y_j},Q^*_jQ_j]\cdot\grad_{y_j})
\\
=&-i\eps(P_j^*Q_j+Q_j^*P_j)\,.
\ea
$$
On the other hand
$$
\ba
\left[V_{jk}\otimes I_{\fH}+I_{\fH}\otimes V_{jk},P_j^*P_j\right]=&[V_{jk}\otimes I_{\fH}+I_{\fH}\otimes V_{jk},P_j^*]P_j
\\
&+P_j^*[V_{jk}\otimes I_{\fH}+I_{\fH}\otimes V_{jk},P_j]
\\
=&i\eps(\grad V(x_j-x_k)-\grad V(y_j-y_k))P_j
\\
&+i\eps P_j^*(\grad V(x_j-x_k)-\grad V(y_j-y_k)))\,.
\ea
$$

Therefore 
$$
\ba
\frac1\eps\Tr\left(i\left[\cH_{\eps,N}\otimes I_{\fH_N}+ I_{\fH_N}\otimes\cH_{\eps,N},\frac1N\sum_{j=1}^N(Q^*_jQ_j+P^*_jP_j)\right]R_{\eps,N}\right)&
\\
=\Tr\left(\frac1N\sum_{j=1}^N(P_j^*Q_j+Q_j^*P_j)R_{\eps,N}\right)&
\\
-\Tr\left(\frac1{N^2}\sum_{1\le j,k\le N}P_j^*(\grad V(x_j-x_k)-\grad V(y_j-y_k))R_{\eps,N}\right)&
\\
-\Tr\left(\frac1{N^2}\sum_{1\le j,k\le N}(\grad V(x_j-x_k)-\grad V(y_j-y_k))P_jR_{\eps,N}\right)&\,.
\ea
$$
By the Cauchy-Schwarz inequality
$$
|\Tr\left((P_j^*Q_j+Q_j^*P_j)R_{\eps,N}\right)|\le\Tr\left((P_j^*P_j+Q_j^*Q_j)R_{\eps,N}\right)\,,
$$
while
$$
\ba
\Tr((P_j^*(\grad V(x_j-x_k)\!-\!\grad V(y_j-y_k))\!+\!(\grad V(x_j-x_k)\!-\!\grad V(y_j-y_k))P_j)R_{\eps,N})&
\\
\le\Tr(P_j^*P_jR_{\eps,N})+\Tr(|\grad V(x_j-x_k)-\grad V(y_j-y_k)|^2R_{\eps,N})&\,.
\ea
$$

Eventually
$$
\ba
\frac1\eps\Tr\left(i\left[\cH_{\eps,N}\otimes I_{\fH_N}+ I_{\fH_N}\otimes\cH_{\eps,N},\frac1N\sum_{j=1}^N(Q^*_jQ_j+P^*_jP_j)\right]R_{\eps,N}\right)&
\\
\le
\Tr\left(\frac1N\sum_{j=1}^N(P_j^*P_j+Q_j^*Q_j)R_{\eps,N}\right)+\Tr\left(\frac1{N^2}\sum_{1\le j,k\le N}P_j^*P_jR_{\eps,N}\right)&
\\
+\Tr\left(\frac1{N^2}\sum_{1\le j,k\le N}|\grad V(x_j-x_k)-\grad V(y_j-y_k)|^2R_{\eps,N}\right)&
\\
\le2D_{\eps,N}+\Lip(\grad V)^2\Tr\left(\frac1{N^2}\sum_{1\le j,k\le N}|(x_j-x_k)-(y_j-y_k)|^2R_{\eps,N}\right)&
\\
\le2D_{\eps,N}+\Lip(\grad V)^2\Tr\left(\frac2{N^2}\sum_{1\le j,k\le N}(|x_j-y_j|^2+|x_k-y_k|^2)R_{\eps,N}\right)&
\\
=2D_{\eps,N}+\Lip(\grad V)^2\Tr\left(\frac4{N}\sum_{l=1}^N|x_l-y_l|^2R_{\eps,N}\right)&
\\
\le(2+4\Lip(\grad V)^2)D_{\eps,N}&\,,
\ea
$$
so that
$$
\ba
\frac{dD_{\eps,N}}{dt}\le(2+4\Lip(\grad V)^2)D_{\eps,N}&
\\
+\frac1\eps\Tr\left(i\left[\bH_{\eps,N}^{\rho_\eps}\otimes I_{\fH_N}-\cH_{\eps,N}\otimes I_{\fH_N},\frac1N\sum_{j=1}^N(Q^*_jQ_j+P^*_jP_j)\right]R_{\eps,N}\right)&\,.
\ea
$$

Next observe that
$$
\ba
\left[\bH_{\eps,N}^{\rho_\eps}\otimes I_{\fH_N}-\cH_{\eps,N}\otimes I_{\fH_N},\frac1N\sum_{j=1}^N(Q^*_jQ_j+P^*_jP_j)\right]&
\\
=\frac1N\sum_{j=1}^N[V_{\rho_\eps}(x_j),P^*_jP_j]-\frac1{N^2}\sum_{1\le k<l\le N}[V(x_k-x_l),(P^*_kP_k+P^*_lP_l)]&
\\
=\frac1N\sum_{j=1}^N([V_{\rho_\eps}(x_j),P^*_j]P_j+P^*_j[V_{\rho_\eps}(x_j),P_j])&
\\
-\frac1{N^2}\sum_{1\le k<l\le N}([V(x_k-x_l),P^*_k]P_k+P^*_k[V(x_k-x_l),P_k])&
\\
-\frac1{N^2}\sum_{1\le k<l\le N}([V(x_k-x_l),P^*_l]P_l+P^*_l[V(x_k-x_l),P_l])&
\\
=\frac1N\sum_{j=1}^N([V_{\rho_\eps}(x_j),P^*_j]P_j+P^*_j[V_{\rho_\eps}(x_j),P_j])&
\\
-\frac1{N^2}\sum_{1\le k,l\le N}([V(x_k-x_l),P^*_k]P_k+P^*_k[V(x_k-x_l),P_k])&\,.
\ea
$$
Moreover
$$
[V_{\rho_\eps}(x_j),P^*_j]P_j+P^*_j[V_{\rho_\eps}(x_j),P_j]=i\eps(\grad V_{\rho_\eps}(x_j)P_j+P_j^*\grad V_{\rho_\eps}(x_j))\,,
$$
while
$$
\ba
\left[V(x_k-x_l),P^*_k\right]P_k+P^*_k[V(x_k-x_l),P_k]&
\\
=i\eps(\grad V(x_k-x_l)P_k+P^*_k\grad V(x_k-x_l))&\,.
\ea
$$
Therefore
$$
\ba
\frac{i}{\eps}\left[\bH_{\eps,N}^{\rho_\eps}\otimes I_{\fH_N}-\cH_{\eps,N}\otimes I_{\fH_N},\frac1N\sum_{j=1}^N(Q^*_jQ_j+P^*_jP_j)\right]&
\\
=
-\frac1{N}\sum_{j=1}^N(\grad V_{\rho_\eps}(x_j)P_j+P_j^*\grad V_{\rho_\eps}(x_j))&
\\
+\frac1{N^2}\sum_{1\le k,l\le N}(\grad V(x_k-x_l)P_k+P^*_k\grad V(x_k-x_l))&
\\
=-\frac1{N}\sum_{j=1}^N\left(\grad V_{\rho_\eps}(x_j)-\frac1N\sum_{k=1}^N\grad V(x_j-x_k)\right)P_j&
\\
-\frac1{N}\sum_{j=1}^NP^*_j\left(\grad V_{\rho_\eps}(x_j)-\frac1N\sum_{k=1}^N\grad V(x_j-x_k)\right)&\,.
\ea
$$

Thus
$$
\ba
\frac1\eps\Tr\left(i\left[\bH_{\eps,N}^{\rho_\eps}\otimes I_{\fH_N}-\cH_{\eps,N}\otimes I_{\fH_N},\frac1N\sum_{j=1}^N(Q^*_jQ_j+P^*_jP_j)\right]R_{\eps,N}\right)&
\\
=
-\frac1{N}\sum_{j=1}^N\Tr\left(\left(\grad V_{\rho_\eps}(x_j)-\frac1N\sum_{k=1}^N\grad V(x_j-x_k)\right)P_jR_{\eps,N}\right)&
\\
-\frac1{N}\sum_{j=1}^N\Tr\left(P^*_j\left(\grad V_{\rho_\eps}(x_j)-\frac1N\sum_{k=1}^N\grad V(x_j-x_k)\right)R_{\eps,N}\right)&
\\
\le
\frac1{N}\sum_{j=1}^N\Tr\left(\left|\grad V_{\rho_\eps}(x_j)-\frac1N\sum_{k=1}^N\grad V(x_j-x_k)\right|^2R_{\eps,N}\right)&
\\
+\frac1{N}\sum_{j=1}^N\Tr\left(P^*_jP_jR_{\eps,N}\right)&\,,
\ea
$$
so that
$$
\ba
\frac1\eps\Tr\left(i\left[\bH_{\eps,N}^{\rho_\eps}\otimes I_{\fH_N}-\cH_{\eps,N}\otimes I_{\fH_N},\frac1N\sum_{j=1}^N(Q^*_jQ_j+P^*_jP_j)\right]R_{\eps,N}\right)&
\\
\le D_{\eps,N}+\frac1{N}\sum_{j=1}^N\Tr\left(\left|\grad V_{\rho_\eps}(x_j)-\frac1N\sum_{k=1}^N\grad V(x_j-x_k)\right|^2R_{\eps,N}\right)&\,.
\ea
$$

Therefore, the differential inequality satisfied by $D_{\eps,N}$ is recast as
$$
\ba
\frac{dD_{\eps,N}}{dt}&\le(3+4\Lip(\grad V)^2)D_{\eps,N}
\\
&+\frac1{N}\sum_{j=1}^N\Tr\left(\left|\grad V_{\rho_\eps}(x_j)-\frac1N\sum_{k=1}^N\grad V(x_j-x_k)\right|^2R_{\eps,N}\right)
\\
&=(3+4\Lip(\grad V)^2)D_{\eps,N}
\\
&+\frac1{N}\sum_{j=1}^N\Tr\left(\left|\grad V_{\rho_\eps}(x_j)-\frac1N\sum_{k=1}^N\grad V(x_j-x_k)\right|^2\rho_\eps^{\otimes N}\right)\,.
\ea
$$
The last inequality comes from the fact that the multiplication by
$$
\left|\grad V_{\rho_\eps}(x_j)-\frac1N\sum_{k=1}^N\grad V(x_j-x_k)\right|^2
$$
acts on the variables variables $x_1,\ldots,x_N$ only, and that the first marginal $R_{\eps,N,1}$ de $R_{\eps,N}$ is $\rho_\eps^{\otimes N}$ by Lemma \ref{L-RN}.

Let us finally estimate the term
$$
\frac1{N}\sum_{j=1}^N\Tr\left(\left|\grad V_{\rho_\eps}(x_j)-\frac1N\sum_{k=1}^N\grad V(x_j-x_k)\right|^2\rho_\eps^{\otimes N}\right)\,.
$$ 
By a straightforward computation
$$
\ba
\Tr\left(\left|\grad V_{\rho_\eps}(x_j)-\frac1N\sum_{k=1}^N\grad V(x_j-x_k)\right|^2\rho_\eps^{\otimes N}\right)&
\\
=\int\left|\int\grad V(x_j-x')\rho_\eps(t,x',x')dx'-\frac1N\sum_{k=1}^N\grad V(x_j-x_k)\right|^2\prod_{\ell=1}^N\rho_\eps(t,x_\ell,x_\ell)dX_N&\,,
\ea
$$
Set 
$$
f_\eps(t,x):=\rho_\eps(t,x,x)\,,
$$
where $\rho_\eps$ is the solution of the Hartree equation (\ref{HartreeRho}). Since $\rho_\eps(t)\in\cD(\fH)$, then $f_\eps(t)$ is a probability density on $\bR^d$ for each $\eps>0$ and each $t\ge 0$. Applying Lemma \ref{L-CombIneq} with
$p=2$ and $\rho=f_\eps(t,\cdot)$ shows that (see formula (\ref{Ineq1}) below)
$$
\Tr\left(\left|\grad V_{\rho_\eps}(x_j)-\frac1N\sum_{k=1}^N\grad V(x_j-x_k)\right|^2\rho_\eps^{\otimes N}\right)\le\frac{8}N\|\grad V\|_{L^\infty}^2\,.
$$

Therefore
$$
\frac{dD_{\eps,N}}{dt}(t)\le\L D_{\eps,N}(t)+\frac{8}N\|\grad V\|_{L^\infty}^2
$$
with
$$
\Lambda:=3+4\Lip(\grad V)^2\,.
$$
Gronwall's inequality implies that
$$
D_{\eps,N}(t)\le D_{\eps,N}(0)e^{\L t}+\frac{8}N\|\grad V\|_{L^\infty}^2\frac{e^{\L t}-1}{\L}\,.
$$

\subsection{From $D_{\eps,N}$ to $MK_2^\eps(\rho_\eps(t)^{\otimes N},\rho_{\eps,N}(t))$}


Let us denote by $R^\mathbf{n}_{\eps,N}$ the marginal density of $R_{\eps,N}$ corresponding to the $n$ first particles. In other words, for each $A\in\cL(\fH_{2n})$, we set
$$
\Tr_{\fH_{2n}}(AR^\mathbf{n}_{\eps,N})=\Tr_{\fH_{2N}}(S(A\otimes I_{\fH_{2(N-n)}})S^*R_{\eps,N})
$$
where
$$
S\Psi(x_1,y_1,x_2,y_2,\ldots,x_N,y_N)=\Psi(x_1,\ldots,x_N,y_1,\ldots,y_N)\,.
$$

We claim that
\be\lb{De1}
R^\mathbf{n}_{\eps,N}\in\cQ(\rho_\eps^{\otimes n},\rho^\mathbf{n}_{\eps,N})\,.
\ee
Indeed, for each $B\in\cL(\fH_n)$, by Lemma \ref{L-RN}, one has
$$
\ba
\Tr_{\fH_{2n}}((B\otimes I_{\fH_n})R^\mathbf{n}_{\eps,N})=\Tr_{\fH_{2N}}(S(B\otimes I_{\fH_n}\otimes I_{\fH_{2(N-n)}})S^*R_{\eps,N})&
\\
=\Tr_{\fH_{2N}}((B\otimes I_{\fH_{N-n}})\otimes I_{\fH_N})R_{\eps,N})&
\\
=\Tr_{\fH_{2N}}((B\otimes I_{\fH_{N-n}})R_{\eps,N,1})&
\\
=\Tr_{\fH_N}((B\otimes I_{\fH_{N-n}})\rho_\eps^{\otimes N})&
\\
=\Tr_{\fH_n}(B\rho_\eps^{\otimes n})&\,,
\ea
$$
while
$$
\ba
\Tr_{\fH_{2n}}((I_{\fH_n}\otimes B)R^\mathbf{n}_{\eps,N})=\Tr_{\fH_{2N}}(S(I_{\fH_n}\otimes B\otimes I_{\fH_{2(N-n)}})S^*R_{\eps,N})&
\\
=\Tr_{\fH_{2N}}((I_{\fH_N}\otimes(B\otimes I_{\fH_{N-n}}))R_{\eps,N})&
\\
=\Tr_{\fH_{2N}}((B\otimes I_{\fH_{N-n}})R_{\eps,N,2})&
\\
=\Tr_{\fH_N}((B\otimes I_{\fH_{N-n}})\rho_{\eps,N})&
\\
=\Tr_{\fH_n}(B\rho^\mathbf{n}_{\eps,N})&\,.
\ea
$$

Observe further that, for each $j=1,\ldots,N$, one has
$$
\ba
\Tr_{\fH_{2N}}((Q_j^*Q_j+P^*_jP_j)R_{\eps,N})=&\Tr_{\fH_{2N}}(\cT^*_{\si_j}(Q_1^*Q_1+P^*_1P_1)\cT_{\si_j}R_{\eps,N})&
\\
=&\Tr_{\fH_{2N}}((Q_1^*Q_1+P^*_1P_1)\cT_{\si_j}R_{\eps,N}\cT^*_{\si_j})
\\
=&\Tr_{\fH_{2N}}((Q_1^*Q_1+P^*_1P_1)R_{\eps,N})\,,
\ea
$$
where $\si_j$ is the permutation of $\{1,\ldots,N\}$ exchanging $1$ and $j$ and leaving all the other integers invariant, because of the identity (\ref{SymReNt}) in Lemma \ref{L-RN}. 

Hence, for each $N\ge n\ge 1$, each $\eps>0$ and each $t\ge 0$, one has
$$
\ba
D_{\eps,N}(t)&=\Tr_{\fH_{2N}}((Q^*_jQ_j+P^*_jP_j)R_{\eps,N}(t))\qquad\hbox{ for each }j=1,\ldots,N\,,
\\
&=\frac1n\Tr_{\fH_{2N}}\left(\sum_{j=1}^n(Q^*_jQ_j+P^*_jP_j)R_{\eps,N}(t)\right)
\\
\\
&=\frac1n\Tr_{\fH_{2n}}\left(\sum_{j=1}^n(Q^*_jQ_j+P^*_jP_j)R^\mathbf{n}_{\eps,N}(t)\right)\,.
\ea
$$
Because of (\ref{De1}), this implies that
$$
D_{\eps,N}(t)\ge\frac1nMK^\eps_2(\rho_\eps(t)^{\otimes n},\rho^\mathbf{n}_{\eps,N}(t))\,,
$$
for each $N\ge n\ge 1$, each $\eps>0$ and each $t\ge 0$. In particular, 
\be\lb{1/2FinalDobIneq}
\ba
MK^\eps_2(\rho_\eps(t)^{\otimes n},\rho^\mathbf{n}_{\eps,N}(t))^2\le &nD_{\eps,N}(0)e^{\L t}+\frac{8n}{N}\|\grad V\|_{L^\infty}^2\frac{e^{\L t}-1}{\L}\,.
\ea
\ee
This inequality holds for each $R_{\eps,N}^{in}\in\cQ((\rho_\eps^{in})^{\otimes N},\rho_{\eps,N}^{in})$ satisfying the symmetry condition (\ref{SymReNin}). 

\subsection{Upper bound for $D_{\eps,N}(0)$}


Let $Q_{\eps,N}^{in}\in\cQ((\rho_\eps^{in})^{\otimes N},\rho_{\eps,N}^{in})$, and set 
$$
R_{\eps,N}^{in}:=\frac1{N!}\sum_{\si\in\fS_N}\cT_\si Q_{\eps,N}^{in}\cT^*_\si\,.
$$
By construction, $R_{\eps,N}^{in}\in\cD(\fH_N\otimes\fH_N)$ and satisfies (\ref{SymReNin}). On the other hand, for each $A\in\cL(\fH_N)$, one has
$$
\ba
\Tr_{\fH_N\otimes\fH_N}((I_{\fH_N}\otimes A)R_{\eps,N}^{in})=&\frac1{N!}\sum_{\si\in\fS_N}\Tr_{\fH_N\otimes\fH_N}((I_{\fH_N}\otimes A)\cT_\si Q_{\eps,N}^{in}\cT^*_\si)
\\
=&\frac1{N!}\sum_{\si\in\fS_N}\Tr_{\fH_N\otimes\fH_N}(\cT^*_\si(I_{\fH_N}\otimes A)\cT_\si Q_{\eps,N}^{in})
\\
=&\frac1{N!}\sum_{\si\in\fS_N}\Tr_{\fH_N\otimes\fH_N}((I_{\fH_N}\otimes(\tau^*_\si A\tau_\si))Q_{\eps,N}^{in})
\\
=&\frac1{N!}\sum_{\si\in\fS_N}\Tr_{\fH_N}(\tau^*_\si A\tau_\si\rho_{\eps,N}^{in})
\\
=&\frac1{N!}\sum_{\si\in\fS_N}\Tr_{\fH_N}(A\tau_\si\rho_{\eps,N}^{in}\tau^*_\si)
\\
=&\Tr_{\fH_N}(A\rho_{\eps,N}^{in})
\ea
$$
since $\rho_{\eps,N}^{in}$ satisfies the symmetry condition (\ref{SymDensIn}). By the same token
$$
\Tr_{\fH_N\otimes\fH_N}((A\otimes I_{\fH_N})R_{\eps,N}^{in})=\Tr_{\fH_N}(A(\rho_\eps^{in})^{\otimes N})
$$
for each $A\in\cL(\fH_N)$. Hence $R_{\eps,N}^{in}\in\cQ((\rho_\eps^{in})^{\otimes N},\rho_{\eps,N}^{in})$.

On the other hand
$$
\ba
D_{\eps,N}(0)=&\Tr_{\fH_{2N}}\left(\frac1N\sum_{j=1}^N(Q^*_jQ_j+P^*_jP_j)R_{\eps,N}^{in}\right)
\\
=&\Tr_{\fH_{2N}}\left(\frac1N\sum_{j=1}^N(Q^*_jQ_j+P^*_jP_j)\frac1{N!}\sum_{\si\in\fS_N}\cT_\si Q_{\eps,N}^{in}\cT^*_\si\right)
\\
=&\Tr_{\fH_{2N}}\left(\frac1{N!}\sum_{\si\in\fS_N}\cT^*_\si\left(\frac1N\sum_{j=1}^N(Q^*_jQ_j+P^*_jP_j)\right)\cT_\si Q_{\eps,N}^{in}\right)
\\
=&\frac1N\Tr_{\fH_{2N}}\left(\sum_{j=1}^N(Q^*_jQ_j+P^*_jP_j)Q_{\eps,N}^{in}\right)
\ea
$$
because
$$
\cT^*_\si\left(\frac1N\sum_{j=1}^N(Q^*_jQ_j+P^*_jP_j)\right)\cT_\si=\frac1N\sum_{j=1}^N(Q^*_jQ_j+P^*_jP_j)
$$
for each $\si\in\fS_N$.

Inserting this in (\ref{1/2FinalDobIneq}) shows that
$$
\ba
\frac1nMK^\eps_2(\rho_\eps(t)^{\otimes n},\rho^\mathbf{n}_{\eps,N}(t))^2\le&\frac{8}{N}\|\grad V\|_{L^\infty}^2\frac{e^{\L t}-1}{\L}
\\
&+\frac{e^{\L t}}N\Tr_{\fH_{2N}}\left(\sum_{j=1}^N(Q^*_jQ_j+P^*_jP_j)Q_{\eps,N}^{in}\right)
\ea
$$
for all $Q_{\eps,N}^{in}\in\cQ((\rho_\eps^{in})^{\otimes N},\rho_{\eps,N}^{in})$ with $N\ge n\ge 1$. Minimizing over $Q_{\eps,N}^{in}$ leads to 
$$
\frac1nMK^\eps_2(\rho_\eps(t)^{\otimes n},\rho^\mathbf{n}_{\eps,N}(t))^2\le\frac{8}{N}\|\grad V\|_{L^\infty}^2\frac{e^{\L t}-1}{\L}+\frac{e^{\L t}}NMK_2^\eps((\rho_\eps^{in})^{\otimes N},\rho_{\eps,N}^{in})
$$
which is precisely inequality (\ref{MainIneq}).

Next, in the special case where $\rho_\eps^{in}=\Op^T_\eps((2\pi\eps)^d\mu_\eps^{in})$ with $\mu_\eps^{in}\in\cP_2(\bC^d)$, while $\rho_{\eps,N}^{in}=\Op^T_\eps((2\pi\eps)^{dN}\mu_{\eps,N}^{in})$ with $\mu_{\eps,N}^{in}\in\cP_2((\bC^d)^N)$
satisfies (\ref{SymDensIn}), one has $(\rho_\eps^{in})^{\otimes N}=\Op^T_\eps((2\pi\eps)^{dN}(\mu_\eps^{in})^{\otimes N})$, and we deduce from Theorem \ref{T-MKeps} (2) that
$$
MK_2^\eps((\rho_\eps^{in})^{\otimes N},\rho_{\eps,N}^{in})^2\le\MKd((\mu_\eps^{in})^{\otimes N},\mu_{\eps,N}^{in})^2+2Nd\eps\,.
$$
Inserting this in (\ref{MainIneq}) shows that
$$
\ba
\frac1nMK^\eps_2(\rho_\eps(t)^{\otimes n},\rho^\mathbf{n}_{\eps,N}(t))^2\le&\frac{8}{N}\|\grad V\|_{L^\infty}^2\frac{e^{\L t}-1}{\L}
\\
&+e^{\L t}\left(\frac1N\MKd((\mu_\eps^{in})^{\otimes N},\mu_{\eps,N}^{in})^2+2d\eps\right)
\ea
$$
which is (\ref{MainIneq2}). This concludes the proof of Theorem \ref{T-UQMF}.

\begin{Rmk}
Let $\rho_\eps$ be the solution of the Cauchy problem for the Hartree equation with initial data $\rho^{in}_\eps\in\cD(\fH)$ such that
$$
\Tr((|x|^2-\eps^2\Dlt)\rho_\eps^{in})<\infty\,.
$$
Then, for each $\eps>0$ and each $t\ge 0$, one has
$$
\Tr(|x|^2\rho_\eps(t))\le(\Tr(|x|^2\rho_\eps^{in})+te^t(\Tr(-\eps^2\Dlt\rho_\eps^{in})+2\|V\|_{L^\infty})<\infty\,,
$$
as a consequence of the differential inequality
$$
\left|\frac{d}{dt}\Tr\left(|x|^2\rho_\eps(t)\right)\right|=\left|\Tr\left([\tfrac12\eps\Dlt,|x|^2]\rho_\eps(t)\right)\right|\le\Tr\left(\left(|x|^2-\eps^2\Dlt\right)\rho_\eps(t)\right)
$$
and of the energy conservation
$$
\frac{d}{dt}\left(\Tr(-\eps^2\Dlt\rho_\eps(t))+\iint V(x-z)\rho_\eps(t,z,z)\rho_\eps(t,x,x)dxdz\right)=0\,.
$$
In particular, for each $t\ge 0$ and each $\eps>0$, one has
$$
\Tr((|x|^2-\eps^2\Dlt)\rho_\eps(t))<\infty\,.
$$
\end{Rmk}


\begin{appendix}


\section{Proof of Lemma \ref{L-CombIneq}}\lb{A-CombIneq}


Without loss of generality, we treat the case $j=1$. Define
$$
\cV(z):=F\star\rho(x_1)-F(x_1-z)\,,\qquad z\in\bR^d\,.
$$

\subsection{The case where $p\ge 2$ is an even integer.}

Start with the decomposition
$$
\ba
\int\left|F\star\rho(x_1)-\frac1N\sum_{k=1}^NF(x_1-x_k)\right|^{p}\prod_{m=1}^N\rho(x_m)dx_m&
\\
=\int\left|\frac1N\sum_{i=1}^N\cV(x_i)\right|^p\prod_{m=1}^N\rho(x_m)dx_m&
\\
=\frac1{N^p}\sum_{j,k\in\{1,\ldots,N\}^{\{1,\ldots,p/2\}}}\int\prod_{l=1}^{p/2}\cV(x_{j(l)})\cdot\cV(x_{k(l)})\prod_{m=1}^N\rho(x_m)dx_m&\,.
\ea
$$
For each pair of maps $j,k\!\in\!\{1,\ldots,N\}^{\{1,\ldots,p/2\}}$, define the map $g\!\in\!\{1,\ldots,N\}^{\{1,\ldots,p\}}$ as follows:
$$
g(l)=j(l)\hbox{ if }1\le l\le p/2\,,\quad g(l)=k(l-p/2)\hbox{ if }p/2+1\le l\le p\,.
$$
Conversely, given $g\in\{1,\ldots,N\}^{\{1,\ldots,p\}}$, one reconstructs the maps $j,k$ by the formulas
$$
j(l)=g(l)\,,\quad k(l)=g(l+p/2)\,,\qquad 1\le l\le p/2\,.
$$

Hence
$$
\ba
\int\left|F\star\rho(x_1)-\frac1N\sum_{k=1}^NF(x_1-x_k)\right|^p\prod_{m=1}^N\rho(x_m)dx_m&
\\
=\frac1{N^p}\sum_{g\in\{1,\ldots,N\}^{\{1,\ldots,p\}}}\int\prod_{l=1}^{p/2}\cV(x_{g(l)})\cdot\cV(x_{g(l+p/2)})\prod_{m=1}^N\rho(x_m)dx_m&
\\
=\frac1{N^p}\sum_{g\in S_N^p}\int\prod_{l=1}^{p/2}\cV(x_{g(l)})\cdot\cV(x_{g(l+p/2)})\prod_{m=1}^N\rho(x_m)dx_m&
\\
+\frac1{N^p}\sum_{g\in M_N^p}\int\prod_{l=1}^{p/2}\cV(x_{g(l)})\cdot\cV(x_{g(l+p/2)})\prod_{m=1}^N\rho(x_m)dx_m&\,,
\ea
$$
where
$$
\ba
M_N^p&:=\{g\in\{1,\ldots,N\}^{\{1,\ldots,p\}}\hbox{ s.t. }\#g^{-1}(\{m\})\not=1\hbox{ for each }m=2,\ldots,N\}\,,
\\
S_N^p&:=\{1,\ldots,N\}^{\{1,\ldots,p\}}\setminus M_N^p\,.
\ea
$$

Obviously, for each $m\in\{1,\ldots,N\}^{\{1,\ldots,p\}}$, one has
$$
\left|\int\prod_{l=1}^{p/2}\cV(x_{g(l)})\cdot\cV(x_{g(l+p/2)})\prod_{m=1}^N\rho(x_m)dx_m\right|\le(2\|F\|_{L^\infty})^p\,.
$$
Next, for each $g\in S_N^p$, define
\be\lb{Deflk}
\ba
{}&l_g:=\min\{l=1,\ldots,p\hbox{ s.t. }g(l)>1\hbox{ and }\#g^{-1}(\{g(l))\})=1\}\,,
\\
&\hat l_g=l_g+p/2\quad\hbox{ if }l_g\le p/2\,,\qquad\hbox{ and }\hat l_g=l_g-p/2\quad\hbox{ if }l_g>p/2\,.
\ea
\ee

Thus, if $g\in S_N^p$, one has
$$
\ba
\int\prod_{l=1}^{p/2}\cV(x_{g(l)})\cdot\cV(x_{g(l+p/2)})\prod_{m=1}^N\rho(x_m)dx_m&
\\
=\int\left(\prod_{1\le l\le p/2\atop l\notin\{l_g,\hat l_g\}}\cV(x_{g(l)})\cdot\cV(x_{g(p/2+l)})\right)\cV(x_{g(\hat l_g)})\prod_{1\le m\le N\atop m\not=g(l_g)}\rho(x_m)dx_m&
\\
\cdot\int\cV(x_{g(l_g)})\rho(x_{g(l_g)})dx_{g(l_g)}=0&\,,
\ea
$$
since
$$
\ba
\int\cV(x_{g(l_g)})\rho(x_{g(l_g)})dx_{g(l_g)}=\int(F\star\rho(x_1)-F(x_1-x_{g(l_g)}))\rho(x_{g(l_g)})dx_{g(l_g)}&
\\
=F\star\rho(x_1)-\int F(x_1-x_{g(l_g)}))\rho(x_{g(l_g)})dx_{g(l_g)}=0&\,.
\ea
$$

Hence
\be\lb{Ineq0}
\int\left|F\star\rho(x_j)-\frac1N\sum_{k=1}^NF(x_1-x_k)\right|^p\prod_{m=1}^N\rho(x_m)dx_m
\le
\frac{\#M_N^p}{N^p}(2\|F\|_{L^\infty})^p\,.
\ee

Now 
$$
\#M_N^p=N^{2p}-\#S_N^p\,,
$$
and we next compute $\#S_N^p$. 

For each element $g\in S_N^p$, there are $N-1$ choices for $g(l_g)\in\{2,N\}$, where $l_g$ is the integer defined in (\ref{Deflk}). For each such choice, the restriction of the map $g$ to $\{1,\ldots,p\}\setminus\{l_g\}$ takes its values in 
$\{1,\ldots,N\}\setminus\{g(l_g)\}$ and can be chosen arbitrarily among the maps from $\{1,\ldots,p\}\setminus\{l_g\}$ to $\{1,\ldots,N\}\setminus\{g(l_g)\}$. Hence
$$
\#S_N^p=(N-1)(N-1)^{p-1}=(N-1)^p
$$
so that
$$
\#M_N^p=N^p-(N-1)^p\,.
$$
Hence
$$
\frac{\#M_N^p}{N^p}=1-\left(1-\frac1N\right)^p\le\frac{p}N\,.
$$

Inserting this inequality in (\ref{Ineq0}) shows that
\be\lb{Ineq1}
\int\left|F\star\rho(x_1)-\frac1N\sum_{k=1}^NF(x_1-x_k)\right|^p\prod_{m=1}^N\rho(x_m)dx_m\le\frac{p}{N}(2\|F\|_{L^\infty})^p\,.
\ee

\subsection{The case where $p\ge 2$ is not an even integer}


Write $p/2$ as
$$
p/2=(1-\th)[p/2]+\th([p/2]+1)\,,\qquad\hbox{ with }0<\th<1\,,
$$
where $[z]$ denotes the largest integer less than or equal to $z$. Using H\"older's inequality and (\ref{Ineq1}) shows that
$$
\ba
\int\left|F\star\rho(x_1)-\frac1N\sum_{k=1}^NF(x_1-x_k)\right|^{p}\prod_{m=1}^N\rho(x_m)dx_m&
\\
\le\frac{(2[p/2])^{1-\th}((2[p/2]+2)^\th}{N}(2\|F\|_{L^\infty})^p&
\\
\le\frac{2[p/2]+2}{N}(2\|F\|_{L^\infty})^p&\,.
\ea
$$

\subsection{The case where $0<p<2$}


Using Jensen's inequality and inequality (\ref{Ineq1}) for $p=2$ shows that
$$
\ba
\int\left|F\star\rho(x_1)-\frac1N\sum_{k=1}^NF(x_1-x_k)\right|^{p}\prod_{m=1}^N\rho(x_m)dx_m&
\\
\le\left(\int\left|F\star\rho(x_1)-\frac1N\sum_{k=1}^NF(x_1-x_k)\right|^2\prod_{m=1}^N\rho(x_m)dx_m\right)^{p/2}
\\
\le\frac{2^{p/2}}{N^{p/2}}(2\|F\|_{L^\infty})^p\le\frac{2[p/2]+2}{N^{p/2}}(2\|F\|_{L^\infty})^p&\,.
\ea
$$


\section{T\"oplitz operators, Wigner and Husimi transforms}\lb{A-Toeplitz}


For each $z\in\bC^d$ with $\Re(z)=q$ and $\Im(z)=p$, we denote by $|z,\eps\ra$ the wave function (sometimes referred to as a ``coherent state'') defined by the formula
\be\lb{CohSt}
|z,\eps\ra(x):=(\pi\eps)^{-d/4}e^{-(x-q)^2/2\eps}e^{ip\cdot x/\eps}\,.
\ee
We recall the bra-ket notation: $|z,\eps\ra\la z,\eps|$ designates the orthogonal projection on the line $\bC|z,\eps\ra$ in $L^2(\bC^d)$.

With the normalization above, denoting $\fH:=L^2(\bR^d)$, one has both
$$
\la z,\eps|z,\eps\ra:=\big\||z,\eps\ra\big\|^2_{\fH}=\int_{\bR^d}\big||z,\eps\ra(x)\big|^2dx=1\,,
$$
and
\be\lb{ToepId}
\frac1{(2\pi\eps)^d}\int_{\bC^d}|z,\eps\ra\la z,\eps|dz=I_\fH\,,
\ee
where the integral on the left hand side is to be understood in the weak operator sense.

To each positive or finite Borel measure $\mu$ on $\bC^d$, we define the T\"oplitz operator at scale $\eps$ with symbol $\mu$ by the formula
$$
\Op^T_\eps(\mu):=\frac1{(2\pi\eps)^d}\int_{\bC^d}|z,\eps\ra\la z,\eps|\mu(dz)\,.
$$
This is a possibly unbounded operator on $\fH$, defined by duality by the formula
$$
\ba
\la v|\Op^T_\eps(\mu)u\ra_\fH:=&\frac1{(2\pi\eps)^d}\int\overline{u(x)}\Op^T_\eps(\mu)v(x)dx
\\
=&\frac1{(2\pi\eps)^d}\int_{\bC^d}\la v|z,\eps\ra\la z,\eps|u\ra\mu(dz)
\ea
$$
for all $u,v\in\fH$ such that $z\mapsto\la z,\eps|u\ra$ and $z\mapsto\la v|z,\eps\ra$ belong to $L^2(\bC^d,\mu)$. 

If $\mu$ is a positive measure, then
\be\lb{SAToep}
\Op^T_\eps(\mu)=\Op^T_\eps(\mu)^*\ge 0\,,\quad\Tr(OP^T_\eps(\mu))=\frac1{(2\pi\eps)^d}\int_{\bC^d}\mu(dz)\,.
\ee
If $f\in L^\infty(\bC^d)$, then
$$
\Op^T_\eps(f)=\frac1{(2\pi\eps)^d}\int_{\bC^d}|z,\eps\ra\la z,\eps|f(z)dz\in\cL(\fH)\,,\hbox{ with }\|\Op^T_\eps(f)\|\le\|f\|_{L^\infty}\,.
$$

The following formulas are elementary but fundamental: if $f$ is a quadratic form on $\bR^d$, then
\be\lb{ToepQuad}
\left\{
\ba
{}&\Op^T_\eps(f(q))=(f(x)+\tfrac14\eps(\Dlt f)I_\fH)\,,
\\
&\Op^T_\eps(f(p))=(f(-i\eps\d_x)+\tfrac14\eps(\Dlt f)I_\fH)\,,
\ea
\right.
\ee
where $f(x)$ designates the unbounded operator defined on $L^2(\bR^d)$ by the formula $(f(x)\phi)(x)=f(x)\phi(x)$ .

Let $A$ be an unbounded operator on $L^2(\bR^d)$, and assume that its Schwartz kernel $k_A$ belongs to $\cS'(\bR^d\times\bR^d)$. In other words, $A$ is the linear map from $\cS(\bR^d)$ to $\cS'(\bR^d)$ defined by the formula
$$
\la Au,v\ra_{\cS'(\bR^d),\cS(\bR^d)}=\la k_A,v\otimes u\ra_{\cS'(\bR^d\times\bR^d),\cS(\bR^d\times\bR^d)}\,.
$$
The Wigner transform of $A$ at scale $\eps$ is defined as
\be\lb{Wigner}
W_\eps[A]:=(2\pi)^{-d}\cF_2(k_A\circ J_\eps)\,,\quad\hbox{ where }J_\eps(x,y)=(x+\tfrac12\eps y,x-\tfrac12\eps y)
\ee
and where $\cF_2$ is the partial Fourier transform in the second variable. When $k_A\circ J_\eps$ is an integrable function, one has
$$
W_\eps[A](x,\xi)=(2\pi\eps)^{-d}\int_{\bR^d}e^{-i\xi\cdot y/\eps}k_A(x+\tfrac12y,x-\tfrac12y)dy\,.
$$
In particular, for each $q,p\in\bR^d$, one has
\be\lb{WCohSt}
\ba
W_\eps[|q+ip,\eps\ra\la q+ip,\eps|](x,\xi)&
\\
=(2\pi\eps)^{-d}\int_{\bR^d}(\pi\eps)^{-d/2}e^{-i\xi\cdot y/\eps}e^{ip\cdot y/\eps}e^{-(|x+y/2-q|^2+|x-y/2-q|^2)/2\eps}dy&
\\
=(\pi\eps)^{-d/2}e^{-|x-q|^2/\eps}(2\pi\eps)^{-d}\int_{\bR^d}e^{i(p-\xi)\cdot y/\eps}e^{-|y|^2/4\eps}dy&
\\
=(\pi\eps)^{-d}e^{-(|x-q|^2+|\xi-p|^2)/\eps}&\,,
\ea
\ee
since 
$$
\int_{\bR^d}(\eps/4\pi)^{d/2}e^{-iy\cdot(\xi-p)}e^{-\eps y^2/4}e^{-|x-q|^2/\eps}dy=e^{-|\xi-p|^2/\eps}
$$
(which is the classical formula for the Fourier transform of a Gaussian density). Thus, for each positive or finite Borel measure on $\bR^d\times\bR^d$, one has
\be\lb{WToep}
W_\eps[\Op^T_\eps(\mu)]=\frac1{(2\pi\eps)^d}G^{2d}_{\eps/2}\star\mu\,,
\ee
where $G^n_a$ is the centered Gaussian density on $\bR^n$ with covariance matrix $aI$. With $\mu(dqdp)=dqdp$, or $\mu(dqdp)=f(q)dqdp$, or $\mu(dqdp)=f(p)dqdp$, where $f$ is an arbitrary quadratic form on $\bR^d$, one finds that
\be\lb{WQuad}
\left\{
\ba
{}&W_\eps[I_\fH](x,\xi)=(2\pi\eps)^{-d}\,,
\\
&W_\eps[f(x)](x,\xi)=(2\pi\eps)^{-d}f(x)\,,
\\
&W_\eps[f(-i\eps\d_x)](x,\xi)=(2\pi\eps)^{-d}f(\xi)\,.
\ea
\right.
\ee
Indeed, when $\mu(dqdp)=g(q,p)dqdp$ with $g$ a polynomial with degree at most $m$, one has
\be\lb{ConvGquad}
G^{2d}_{\eps/2}\star_{x,\xi}g=e^{\eps\Dlt_{x,\xi}/4}g=\sum_{0\le n\le m/2}\frac{\eps^n}{4^nn!}\Dlt^n_{x,\xi}g\,.
\ee
Thus, for $g$ of degree $\le 2$, one finds that $G^{2d}_{\eps/2}\star_{x,\xi}g=(1+\tfrac14\eps\Dlt_{x,\xi})g$. Together with (\ref{ToepId}) and (\ref{ToepQuad}), this formula justifies (\ref{WQuad}).

The Husimi transform of an operator $A$ at scale $\eps$ is defined in terms of the Wigner transform of $A$ by the formula
$$
\tilde W_\eps[A]:=G^{2d}_{\eps/2}\star_{x,\xi}W_\eps[A]\,.
$$
Let $R\in\cD(\fH)$; for each $\psi\in L^2(\bR^d)$, one has
$$
\Tr(|\psi\ra\la\psi|R)=(2\pi\eps)^d\iint\overline{W_\eps[\psi](x,\xi)}W_\eps[R](x,\xi)dxd\xi
$$
by the definition (\ref{Wigner}) of the Wigner transform, and Plancherel's identity. Specializing this formula to $\psi=|z,\eps\ra$, one finds that
$$
\ba
\Tr(|z,\eps\ra\la z,\eps|R)&=\la z,\eps|R|z,\eps\ra
\\
&=(2\pi\eps)^d\tilde W_\eps[R](q,p)\,,\quad\hbox{ where }q=\Re(z)\hbox{ and }p=\Im(z)\,.
\ea
$$
More generally, if $\mu$ is a positive or finite Borel measure on $\bC^d$, one deduces from the previous identity and the Fubini theorem that
\be\lb{TrToep}
\Tr(\Op^T_\eps(\mu)R)=\int_{\bC^d}\tilde W_\eps[R](z)\mu(dz)\,.
\ee
(The previous formula is the particular case where $\mu=(2\pi\eps)^d\de_z$.)

In particular, if $R=\Op^T_\eps((2\pi\eps)^d\mu)$ with $\mu\in\cP_2(\bR^d\times\bR^d)$, for each quadratic form $f$ on $\bR^d$, one has
\be\lb{MK2}
\ba
\Tr((f(x)+f(-i\eps\d_x))\Op^T_\eps((2\pi\eps)^d\mu))&
\\
=\iint_{\bR^d\times\bR^d}(f(q)+f(p))\mu(dqdp)+\tfrac12\eps\Dlt f&\,.
\ea
\ee
Indeed, applying (\ref{TrToep}) shows that
$$
\ba
\Tr((f(x)+f(-i\eps\d_x))\Op^T_\eps((2\pi\eps)^d\mu))&
\\
=\Tr(\Op^T_\eps(f(q)+f(p)-\tfrac12\eps(\Dlt f))\Op^T_\eps((2\pi\eps)^d\mu))&
\\
=\iint_{\bR^d\times\bR^d}(G^{2d}_{\eps/2}\star W_\eps[\Op^T_\eps((2\pi\eps)^d\mu)])(q,p)(f(q)+f(p)-\tfrac12\eps(\Dlt f))dqdp&
\\
=\iint_{\bR^d\times\bR^d}W_\eps[\Op^T_\eps((2\pi\eps)^d\mu)](q,p)(G^{2d}_{\eps/2}\star(f(q)+f(p)-\tfrac12\eps(\Dlt f))dqdp&
\\
=\iint_{\bR^d\times\bR^d}(G^{2d}_{\eps/2}\star G^{2d}_{\eps/2}\star(f(q)+f(p)-\tfrac12\eps(\Dlt f))\mu(dqdp)&\,,
\ea
$$
where the last equality follows from (\ref{WToep}). We conclude by observing that 
$$
\ba
G^{2d}_{\eps/2}\star G^{2d}_{\eps/2}\star(f(q)+f(p)-\tfrac12\eps(\Dlt f))&=e^{\eps\Dlt_{q,p}/2}(f(q)+f(p)-\tfrac12\eps(\Dlt f))
\\
&=(f(q)+f(p)+\tfrac12\eps(\Dlt f))
\ea
$$
according to (\ref{ConvGquad}).
\end{appendix}


\smallskip
\noindent
\textbf{Acknowledgements.} We thank N. Fournier A. Guillin for informing us of their newly published results in \cite{FourGuil}, which we used in an earlier version of this paper, and M. Pulvirenti, who brought to our attention the problem studied 
in this work. We also thank the anonymous referees for their suggestions which have improved our main results.


\end{document}